\documentclass[12pt,reqno]{amsart}
\usepackage{amsmath, amssymb, amsthm}
\usepackage{mathtools}
\usepackage{enumitem}
\usepackage{xcolor}
\usepackage{url}
\usepackage[breaklinks]{hyperref}

\newcommand{\Sel}{\operatorname{Sel}}

\renewcommand{\Im}{\operatorname{Im}}

\newcommand{\Gal}{\operatorname{Gal}}

\renewcommand{\Im}{\operatorname{Im}}
\newcommand{\pr}{\operatorname{pr}}
\newcommand{\s}{{\sigma}}

\newcommand{\g}{\gamma}
\newcommand{\G}{\Gamma}
\newcommand{\ZZ}{\mathbb{Z}}
\newcommand{\SL}{\operatorname{SL}}
\newcommand{\x}{\xi}
\newcommand{\z}{\zeta}

\newcommand{\HCal}{{\mathcal{H}}}

\let\amstextb\b  \let\amstextd\d  \let\amstextk\k  \let\amstextl\l
\let\amstextL\L  \let\amstexto\o  \let\amstextO\O  \let\amstextr\r
\let\amstextt\t  \let\amstexti\i  \let\amstexta\a
\renewcommand{\a}{\ifmmode\alpha\else\expandafter\amstexta\fi}
\renewcommand{\b}{\ifmmode\beta\else\expandafter\amstextb\fi}
\renewcommand{\d}{\ifmmode\delta\else\expandafter\amstextd\fi}
\renewcommand{\k}{\ifmmode\kappa\else\expandafter\amstextk\fi}
\renewcommand{\l}{\ifmmode\lambda\else\expandafter\amstextl\fi}
\renewcommand{\L}{\ifmmode\Lambda\else\expandafter\amstextL\fi}
\renewcommand{\o}{\ifmmode\omega\else\expandafter\amstexto\fi}
\renewcommand{\O}{\ifmmode\Omega\else\expandafter\amstextO\fi}
\renewcommand{\r}{\ifmmode\rho\else\expandafter\amstextr\fi}
\renewcommand{\t}{\ifmmode\tau\else\expandafter\amstextt\fi}
\renewcommand{\i}{\ifmmode\infty\else\expandafter\amstexti\fi}

\numberwithin{equation}{section}
\theoremstyle{plain}
\newtheorem{theorem}{Theorem}[section]
\newtheorem{lemma}[theorem]{Lemma}
\newtheorem{remark}[theorem]{Remark}
\newtheorem{definition}[theorem]{Definition}

\newtheorem{corollary}[theorem]{Corollary}
\newtheorem{proposition}[theorem]{Proposition}
\newtheorem{example}[theorem]{Example}
\newcommand{\Ob}{\mathcal{O}}
\newcommand{\Diff}{\mathfrak{d}}   
\newcommand{\nf}{\mathfrak{n}}
\newcommand{\fa}{\mathfrak{a}}
\newcommand{\fb}{\mathfrak{b}}
\newcommand{\pf}{\mathfrak{p}}
\newcommand{\qf}{\mathfrak{q}}
\newcommand{\mf}{\mathfrak{m}}
\newcommand{\Tr}{\operatorname{Tr}}
\newcommand{\Norm}{\mathrm{N}}
\newcommand{\ub}[1]{u(#1)}
\newcommand{\Gz}[1]{\Gamma_0(#1)}
\newcommand{\Mk}[2]{M_k\!\left(\Gamma_0(#1),\,#2\right)}
\newcommand{\Wn}{W_{\nf}}
\newcommand{\Vn}{V_{\nf}}
\newcommand{\kn}{k_{\nf}}

\newcommand{\ord}{\operatorname{ord}}
\newcommand{\il}{\operatorname{il}}
\newcommand{\bA}{{\mathbf A}}
\newcommand{\ba}{{\mathbf a}}
\newcommand{\bff}{{\mathbf f}}
\newcommand{\lf}{\mathfrak{l}}
\newcommand{\hstar}[2]{h\bigl(#1^{\ast},\,#2\bigr)}
\providecommand{\Cl}{\operatorname{Cl}}
\providecommand{\rank}{\operatorname{rank}}
\providecommand{\tors}{\mathrm{tors}}
\providecommand{\QQ}{\mathbb{Q}}

\DeclareFontEncoding{OT2}{}{}
\DeclareFontFamily{OT2}{wncyr}{}
\DeclareFontShape{OT2}{wncyr}{m}{n}{<->wncyr10}{}
\DeclareFontSubstitution{OT2}{wncyr}{m}{n}
\providecommand{\Sha}{\mbox{\usefont{OT2}{wncyr}{m}{n}Sh}}
\theoremstyle{remark}
\newtheorem{remarknum}[theorem]{Remark}
\theoremstyle{plain}

\title{Selmer groups of CM-twists of elliptic curves}

	\author{Olivia Beckwith}
\author{T\`{u}ng Ho\`{a}ng}

\begin{document}
\begin{abstract}
	Let $F$ be a totally real Galois number field of degree $g$ and let $\ell\geq5$ be a prime. Let $E/F$ be an elliptic curve with an $F$-rational point of
	order $\ell$. Under explicit arithmetic and local hypotheses, together with
	an explicit non-vanishing condition modulo $\ell$, we prove that there are
	$
	\gg_{F,E,\ell}\frac{X^{1/(2g)}}{\log X}
	$
	totally negative square classes $d\in F^\times/(F^\times)^2$ with
	$
	\left|\mathrm{N}_{F/\QQ}\bigl(D(F(\sqrt{d})/F)\bigr)\right|<X
	$
	for which $\Sel_\ell(E^d,F)$ is trivial. Such twists have rank zero and trivial $\ell$-part of the Shafarevich--Tate group. The condition is verifiable by a finite computation, which we carry out in an example. We lift the method of James and Ono from $\QQ$ to the totally real setting, combining a theorem of Morrow, which relates the Selmer group of such a twist to the class group of the CM extension $F(\sqrt{d})$, with an indivisibility theorem of Takai for relative class numbers. Morrow assumes that $g\leq5$, that every prime above $2$ has norm $2$, and that a suitable Hecke character of order $\ell$ exists; we remove the first two hypotheses and replace the third by an elementary congruence at the primes of bad reduction. Takai twists only by quadratic Hecke characters, whereas Morrow's conditions are local; we extend the twisting argument to primitive quadratic residue-class characters to make the two results compatible.
	\end{abstract}

%checked!!
	\maketitle
%%%%%%%% 
	\section{Introduction}
	Let $E$ be an elliptic curve over a number field $F$ given by the Weierstrass equation
	\begin{equation*}
		E: y^2 = x^3 + Ax + B .
	\end{equation*}
	For $d \in F^{\times}/(F^{\times})^{2}$, we let $E^d$ denote the elliptic
	curve with Weierstrass equation
	\begin{equation*}
		E^d: y^2 = x^3 + A d^2 x + Bd^3 .
	\end{equation*}
	All of the $E^d$ are isomorphic to $E$ over the quadratic extension
	$F(\sqrt{d})$, but their arithmetic over the base field $F$ can vary with
	$d$.  For $F=\mathbf Q$, Goldfeld conjectured that $50\%$ of the quadratic twists of $E$ have analytic rank $0$ and $50\%$ have analytic rank $1$. In light of the Birch–Swinnerton-Dyer conjecture, the corresponding prediction for Mordell–Weil ranks is that their average is $1/2$. Smith proved the analogous $50$--$50$ distribution of $2^\infty$-Selmer coranks for elliptic curves satisfying certain technical hypotheses \cite{Smith26a, Smith26b}, and subsequently for every elliptic curve over \(\mathbf Q\) \cite{Smith25}. In particular, the Birch--Swinnerton-Dyer conjecture implies Goldfeld’s conjecture~\cite{Smith25}.

	On the other hand, the behavior of quadratic twists over general number
	fields can be subtler than over $\QQ$. In particular, Tim and Vladimir
	Dokchitser \cite{DD} exhibited elliptic curves over number fields for which
	every quadratic twist is expected to have positive rank. This shows that the
	classical $50\%$ rank $0$/$50\%$ rank $1$ distribution predicted by
	Goldfeld's conjecture over $\QQ$ does not extend directly to arbitrary
	number fields. 
	
	Instead of Mordell--Weil rank, one may consider the $\ell$-Selmer rank. In this paper, we let $\ell$ be an odd prime
	and let $E$ be an elliptic curve over $F$. The Kummer sequence gives
	\begin{equation*}
		1\longrightarrow E(F)/\ell E(F)
		\longrightarrow \Sel_{\ell}(E,F)
		\longrightarrow \Sha(E/F)[\ell]
		\longrightarrow 1 .
	\end{equation*}
	In particular, if $\Sel_{\ell}(E,F)$ is trivial then $E(F)$ has rank $0$, it has
	no point of order $\ell$, and $\Sha(E/F)[\ell]$ is trivial as well. Counting the
	twists $E^{d}$ with trivial $\ell$-Selmer group is therefore a way of producing
	twists of rank $0$ together with information about the Shafarevich--Tate
	group. The bounds obtained in this way are of the order $X^{1/(2g)}/\log X$,
	so they exhibit a set of twists of density $0$ and do not bear directly on the
	proportions predicted by Goldfeld's conjecture.
	
	Over $\QQ$ this was studied by James and Ono \cite{JamesOno}. They proved
	that for a modular elliptic curve $E/\QQ$ and every sufficiently large prime
	$\ell$,
	\begin{equation*}
		\#\left\{\,|D|<X\ :\ D \text{ square-free},\ \Sel_{\ell}(E^{D},\QQ)=1\,\right\}
		\;\gg_{E,\ell}\;
		\frac{\sqrt{X}}{\log X},
	\end{equation*}
	and they proved a companion statement covering the primes $\ell\in\{3,5,7\}$ for
	which $E$ may have a rational point of order $\ell$. In that companion statement, the twisting parameter is taken to be negative,
	so that the associated quadratic field is imaginary. The proof then uses a
	theorem of Frey \cite{Frey} relating the Selmer group of the twist to the
	class group of this imaginary quadratic field.
	
	The purpose of this paper is to prove an analogue of the second statement over a
	totally real field. We state and prove our main result as
	Theorem~\ref{thm:lifting} in \S\ref{sec:main}. Since it involves several
	technical assumptions, we give here a simplified version for real quadratic
	base fields.
	\begin{theorem}\label{thm:quadratic-version}
		Let $F/\QQ$ be a real quadratic field, and let
		$\ell\in\{5,7,11,13\}$ be such that $\ell\nmid h(F)$, with $F\neq\QQ(\sqrt5)$ if $\ell=5$. Let $E/F$ be an
		elliptic curve having an $F$-rational point $P$ of order $\ell$, and suppose that $P$ is not contained in
		the kernel of reduction modulo any prime $\lf$ of $F$ above $\ell$ in the sense of \cite{Morrow}; see
		\S\ref{ss:ec-background}.
		
		Suppose that every prime $\pf\mid N(E)$ with $\pf\nmid\ell$ and
			$\ell\nmid\ord_{\pf}(\Delta_{E})$ satisfies $\ell\mid\Norm\pf-1$.

		Suppose moreover that there exists a rational prime $q>C_{F,E,\ell}$, inert in $F$,
		where $C_{F,E,\ell}=(A/2)^{2}$ is the constant of Theorem~\ref{thm:lifting},
		and put
		\begin{equation*}
			K=F(\sqrt{-q}).
		\end{equation*}
		Assume that $K/F$ satisfies the following local conditions:
		\begin{enumerate}[label=\textup{(\arabic*)},leftmargin=*]
			\item 
			Every prime $\qf\mid2$ dividing $N(E)$ ramifies in $K/F$.
			
			\item  If $\lf\mid\ell$ and $\ord_{\lf}(j_E)<0$, then $\lf$ is inert in $K/F$.
			
			\item For every prime $\pf\mid N(E)$ with $\pf\nmid2\ell$,
			\begin{equation*}
				\begin{cases}
					\pf \text{ is inert in } K/F,
					& \text{if } E/F_{\pf}\text{ is a Tate curve},\\[4pt]
					\pf \text{ splits in } K/F,
					& \text{otherwise}.
				\end{cases}
			\end{equation*}
		\end{enumerate}
		
		Assume in addition that
		\begin{equation*}
			\ell\nmid h^{-}(K/F).
		\end{equation*}
		Then
		\begin{equation*}
			\#\left\{
			d\in F^\times/(F^\times)^2:
			\begin{array}{l}
				d\text{ is totally negative},\\[2pt]
				\left|
				\Norm_{F/\QQ}\!\bigl(D(F(\sqrt d)/F)\bigr)
				\right|<X,\\[2pt]
				\Sel_\ell(E^d,F)=1
			\end{array}
			\right\}
			\gg_{F,E,\ell}
			\frac{X^{1/4}}{\log X}.
		\end{equation*}
	\end{theorem}

		%checked!!
		\begin{example}\label{ex:intro}
			Let $F=\QQ(\sqrt2)$ and $\ell=7$, and consider the elliptic curve
			\begin{equation*}
				E:\ y^2+xy+\sqrt2\,y
				=
				x^3+\sqrt2\,x^2-(14+12\sqrt2)x+24+18\sqrt2,
			\end{equation*}
			a global minimal model with
			$E(F)_{\mathrm{tors}}\cong\ZZ/7\ZZ$ and conductor of norm $142$.
			Then
			\begin{equation*}
				\#\left\{
				d\in F^\times/(F^\times)^2:
				\begin{array}{l}
					d\text{ is totally negative},\\[2pt]
					\left|
					\Norm_{F/\QQ}\!\bigl(D(F(\sqrt d)/F)\bigr)
					\right|<X,\\[2pt]
					\Sel_7(E^d,F)=1
				\end{array}
				\right\}
				\gg_{F,E,7}
				\frac{X^{1/4}}{\log X}.
			\end{equation*}
			The verification of the hypotheses of
			Theorem~\ref{thm:quadratic-version} for this curve is given in
			\S\ref{sec:examples}.
		\end{example}
	%checked!!

	The proof combines three ingredients.
	
	The first is the work of Frey and Morrow, which relates the Selmer group of a
	quadratic twist to the class group of the corresponding quadratic extension. In
	\cite{Frey} Frey proved, for elliptic curves over $\QQ$ with a rational point of
	odd prime order $\ell$, a double divisibility relation between the $\ell$-torsion
	of the class group of an imaginary quadratic field and the order of the
	$\ell$-Selmer group of the associated twist. Morrow \cite{Morrow} extended this
	to base fields of degree at most $5$ in which every prime above $2$ has
	norm $2$, under a hypothesis on an auxiliary Hecke character of order $\ell$.
	In \S\ref{sec:selmer}, for totally real Galois base fields, we remove the two restrictions on the base field, using
	Lemma~\ref{lem:dyadic-ram}, and we replace the hypothesis on the Hecke character
	by an elementary congruence at the primes of bad reduction; see
	Proposition~\ref{prop:star}. The consequence we use is stated as Theorem~\ref{thm:morrow2} in
	\S\ref{ss:frey-morrow}: under a list of local conditions at the primes dividing
	the conductor, the group $\Sel_{\ell}(E^{d},F)$ is trivial if and only if
	$\Cl(F(\sqrt{d}))[\ell]$ is trivial.
	
	The second ingredient is an argument for counting
	CM extensions $K/F$ with class number prime to $\ell$ and with prescribed
	local behaviour at finitely many primes. Over $\QQ$, the idea goes back to an older argument of Kohnen and Ono \cite{KohnenOno}, without
	prescribing a fixed pattern of local conditions, using the theory of half-integral weight modular forms. 
	For counting quadratic twists with trivial Selmer groups, James and Ono \cite{JamesOno}
	modified this argument to include local conditions. Prescribed local conditions in class number indivisibility problems over $\mathbb{Q}$ were subsequently studied by Wiles \cite{Wiles} and quantitatively by Beckwith \cite{Beckwith}; see also \cite{BRR} for recent results with arbitrary splitting conditions. Over a totally real Galois field the
	corresponding statement is due to Takai \cite{Takai}, who proved an
	indivisibility theorem for the relative class numbers $h^{-}(K/F)$ of totally
	imaginary quadratic extensions $K/F$. His proof uses Shimura's Hilbert
	modular Eisenstein series of half-integral weight \cite{Sh},\cite{Sh2}, whose
	Fourier coefficients are, up to elementary factors, the numbers
	$h^{-}(F(\sqrt{-2\xi})/F)$. By taking suitable linear combinations of quadratic twists of this series
	and iterating over the prescribed characters, Takai restricts its Fourier
	expansion to those $\xi$ satisfying prescribed conditions
	$\chi_i(2\xi)=\varepsilon_i$. These character conditions encode the desired splitting, inertness, or
	ramification behaviour at the corresponding primes. A key input in Takai's
	theorem is the non-vanishing congruence
	\eqref{eq:takai-congruence}, which must be verified in each application.
	In certain cases the sum in this congruence reduces to a single
	class-number term, making the verification particularly simple.
	
	The third ingredient, which contains the main new technical input of this
	paper, is to adapt Takai's sieve to the local conditions required by
	Morrow's theorem. Two difficulties arise. First, these local conditions
	are naturally encoded by characters of $(\Ob/\nf)^\times$, whereas
	Takai's Theorem~1 is stated for quadratic Hecke characters. We prove that
	the twisting argument extends to primitive quadratic characters of
	$(\Ob/\nf)^\times$. The key point is Lemma~\ref{lem:phase}, which controls
	the half-integral-weight factor of automorphy under the translations
	occurring in the twist. This yields Theorem~\ref{thm:takai2}, a version of
	Takai's theorem for primitive quadratic residue class characters. The second difficulty is dyadic. At primes above $2$, the required
	ramification conditions cannot be imposed by the quadratic residue
	characters used at odd primes. Lemma~\ref{lem:dyadic} constructs, for
	each relevant prime $\qf\mid2$, a primitive quadratic character of
	$(\Ob/\qf^{\,n})^\times$ together with a sign whose prescribed value
	forces $\qf$ to ramify in $F(\sqrt{-2\xi})/F$. This allows the dyadic
	conditions in Morrow's theorem to be incorporated into Takai's sieve.
	Together, these two extensions make it possible to impose all of the
	local conditions needed for the Selmer argument within Takai's framework.
	%checked!!
	
	The remainder of the paper is organized as follows. Section~\ref{sec:twisting}
	develops the extension of Takai's argument, including the dyadic case.
	Section~\ref{sec:selmer} proves
		Theorem~\ref{thm:morrow2}, which improves the theorem of Morrow.
	Section~\ref{sec:main} proves the main theorem, and
	Section~\ref{sec:examples} verifies its hypotheses for the curve in
	Example~\ref{ex:intro}.
	
	%%%%%%%%%checked!!
	\section{Hilbert modular forms}\label{sec:twisting}
	\subsection{Background}\label{ss:background}
	Throughout this section $F$ is a totally real number field of degree
	$g=[F:\mathbb{Q}]$, with maximal order $\Ob=\Ob_F$, different $\Diff=\Diff_F$, and
	discriminant $D(F)$. We write $\ba$ and $\bff$ for the sets of archimedean and
	non-archimedean places of $F$, view each $v\in\ba$ as an embedding
	$F\hookrightarrow\mathbb{R}$, and let $F_v$ be the completion at $v$; $\bA_F=\bA$ denotes
	the adele ring, with archimedean and finite parts $F_{\ba}$ and $F_{\bff}$. Put
	$G_F=\SL_2(F)$, $G_v=\SL_2(F_v)$, $G_{\bA}=\SL_2(\bA)$, $G_\ba=\prod_{v\in \ba}G_v=\SL_2(\mathbb{R})^{\ba},$ and for
	$\a=\begin{psmallmatrix}a&b\\c&d\end{psmallmatrix}$ write $a=a_\a,\dots,d=d_\a$. Let $\HCal=\{z\in\mathbb{C}:\Im z>0\}$. The groups $G_F$ and $G_{\ba}$
	act on $\HCal^{\ba}$ componentwise: for
	$z=(z_v)_{v\in\ba}\in\HCal^{\ba}$ and
	$\a=\begin{psmallmatrix}a&b\\c&d\end{psmallmatrix}$,
	\begin{equation*}
		\a(z)=
		\left(
		\frac{a_vz_v+b_v}{c_vz_v+d_v}
		\right)_{v\in\ba},
	\end{equation*}
	where $x_v=v(x)$ when $\a\in G_F$, and the same notation denotes the
	$v$-component when $\a\in G_{\ba}$.	We set
	\begin{equation*}
		j(\a,z)=\prod_{v\in\ba}(c_{\a,v}z_v+d_{\a,v}),
	\end{equation*}
	where $x_{v}=v(x)$. For $x\in F$ write $e(x)=\exp\!\big(2\pi i\,\Tr(x)\big)$; this additive character of $F$
	is trivial on $\Diff^{-1}$.
	More generally, for $w=(w_v)_{v\in\ba}\in\mathbb{C}^{\ba}$ we put
	$e(w)=\exp\!\big(2\pi i\sum_{v\in\ba}w_v\big)$, so that
	$e(\xi z)=\exp\!\big(2\pi i\sum_{v\in\ba}\xi_v z_v\big)$ for $\xi\in F$,
	$z\in\HCal^{\ba}$.
	
	For fractional ideals $\fb,\mf$ of $F$ set 
	$D[\fb,\mf]=SO_2(\mathbb{R})^{\ba}\prod_{v\in\bff}D_v[\fb,\mf]$, where
	\begin{equation*}
		D_v[\fb,\mf]=\Big\{\begin{psmallmatrix}a&b\\c&d\end{psmallmatrix}\in\SL_2(F_v):
		a,d\in\Ob_v,\ b\in(2\Diff^{-1}\fb)_v,\ c\in(2^{-1}\mf\fb^{-1}\Diff)_v\Big\},
	\end{equation*}
	and put $\Gamma[\fb,\mf]=G_\ba D[\fb,\mf]\cap G_F$, $\Gz{\mf}=\Gamma[\Ob,\mf]$.
	Explicitly, for an integral ideal $\mf$ with $4\Ob\mid\mf$, 
	\begin{equation*}
		\Gz{\mf}=\Big\{\begin{psmallmatrix}a&b\\c&d\end{psmallmatrix}\in\SL_2(F):
		a,d\in\Ob,\ b\in 2\Diff^{-1},\ c\in 2^{-1}\mf\Diff\Big\}.
	\end{equation*}
	
	Let $M_{\bA}$ be the metaplectic group of Weil \cite{Weil} attached to $G_{\bA}$, with
	projection $\pr:M_{\bA}\to G_{\bA}$, and let
	\begin{equation*}
	\mathsf r:G_F\to M_{\bA}
	\end{equation*}
	be as in \cite[(3.1a)]{Sh}. We identify $G_F$ with its image under $\mathsf r$ and write $\g^\ast=\mathsf r(\g)$.
	
	To pass to half-integral weight we use the theta series
	\begin{equation*}
		\theta(z)=\sum_{\xi\in\Ob}e(\xi^2 z/2),\qquad z\in\HCal^{\ba},
	\end{equation*}
	and the factor of automorphy $h(\g,z)=\theta(\g z)/\theta(z)$ for $\g\in\Gz{4\Ob}$
	\cite[\S2.1]{Takai}. By \cite[p.~286]{Sh} it satisfies
	\begin{equation}\label{eq:hsquared}
		h(\g,z)^2=\operatorname{sgn}\!\big(\Norm_{F/\mathbb{Q}}(d_\g)\big)\,
		\psi_0^\ast\!\big(d_\g\,\il(\g)^{-1}\big)\,j(\g,z)
		=\operatorname{sgn}\!\big(\Norm_{F/\mathbb{Q}}(d_\g)\big)\,\psi_0^\ast(d_\g\Ob)\,j(\g,z),
	\end{equation}
	where $\il(\g)=c_\g\Diff^{-1}+d_\g\Ob$, $\psi_0$ is the Hecke character attached to
	$F(\sqrt{-1})/F$, and $\psi_0^\ast(\fa)=\big(\tfrac{F(\sqrt{-1})/F}{\fa}\big)$ is the
	associated ideal character \cite[(3.17)]{Sh}.
	
Following \cite[\S2.1]{Takai} we work with the theta-quotient normalisation of
\cite{Sh}. Let $k$ be an odd positive integer; the weight is the parallel
half-integral weight $k/2$, with factor of automorphy $h(\g,z)^{k}$, and for a
function $f$ on $\HCal^{\ba}$ and $\g\in\Gz{4\Ob}$ one sets
\begin{equation*}
	(f|_{k}[\g])(z)=f(\g z)\,h(\g,z)^{-k}.
\end{equation*}
	
	Let $\mf$ be an integral ideal with $4\Ob\mid\mf$, so that
	$\Gz{\mf}\subseteq\Gz{4\Ob}$ and $f|_{k}[\g]$ is defined for every
	$\g\in\Gz{\mf}$. Let $\psi$ be a finite-order Hecke character of $F$,
	that is, a continuous character
	\begin{equation*}
		\psi:\bA^{\times}\longrightarrow\mathbb{C}^{\times}
	\end{equation*}
	of finite order which is trivial on the image of $F^{\times}$.
	Suppose that the conductor of $\psi$ divides $\mf$. The restriction of
	$\prod_{v\mid\mf}\psi_v$ to $\prod_{v\mid\mf}\Ob_v^\times$ then factors
	through $(\Ob/\mf)^\times$; we denote the resulting residue class
	character by $\psi_{\mf}$ and assume that $\psi_{\mf}(-1)=1$.
	
	%checked
	\begin{definition}\label{def:hmf}
		A \emph{Hilbert modular form of parallel half-integral weight $k/2$,
			level $\mf$, and nebentypus $\psi$} is a holomorphic function
		$f:\HCal^{\ba}\to\mathbb C$, holomorphic at the cusps when
		$F=\mathbb Q$, such that
		\begin{equation*}
			f|_k[\g]=\psi_{\mf}(d_\g)\,f
			\qquad(\g\in\Gz{\mf}).
		\end{equation*}
		We write $\Mk{\mf}{\psi}$ for the space of such forms.
	\end{definition}

	Every $f\in\Mk{\mf}{\psi}$ has a Fourier expansion
	\begin{equation}\label{eq:fourier}
		f(z)=\sum_{\xi\in F}\l(\xi,\Ob)\,e(\xi z/2)
		=\sum_{\xi\in 2^{-1}\Ob_{+}\cup\{0\}}a_\xi\,e(\xi z),
		\qquad
		\l(\xi,\Ob)\neq 0\ \Longrightarrow\ \xi\in\Ob_{\ge 0},
	\end{equation}
thus $a_\xi=\l(2\xi,\Ob)$ and the support lies in $2^{-1}\Ob$ \cite[\S2.1]{Takai}.
Here $\Ob_{\ge0}$ (resp. $\Ob_{+}$) denotes the totally non-negative (resp.
totally positive) elements.

%checked!!
	\subsection{Twisting operators}\label{ss:twist}
	Let $\nf$ be an integral ideal of $\Ob$, and let
	\begin{equation*}
		\chi:(\Ob/\nf)^\times\longrightarrow\mathbb C^\times
	\end{equation*}
	be a primitive character, extended by $0$ to
	$\kn:=\Ob/\nf$. We do not assume that $\chi$ arises from a Hecke
	character.	
	%checked!!
\begin{definition}\label{def:twist}
	Let $f\in\Mk{\mf}{\psi}$ have the Fourier expansion
	\eqref{eq:fourier}, and let $\chi$ be a character of
	$(\Ob/\nf)^\times$. Define
	\begin{equation*}
		(f\otimes\chi)(z)
		:=
		\sum_{\xi\in2^{-1}\Ob_{+}\cup\{0\}}
		\chi(2\xi)\,a_\xi\,e(\xi z).
	\end{equation*}
\end{definition}
This agrees with Takai's definition
\begin{equation*}
	(f\otimes\chi)(z)
	=
	\sum_{\eta\in F}
	\chi(\eta)\,\l(\eta,\Ob)\,e(\eta z/2),
\end{equation*}
since $\l(\eta,\Ob)=0$ unless $\eta\in\Ob_{\ge0}$, and the change of
variables $\eta=2\xi$ gives $a_\xi=\l(2\xi,\Ob)$.
	
	We next choose the additive character appearing in the twisting formula.
	Set
	\begin{equation*}
		\Wn:=2\Diff^{-1}\nf^{-1},
		\qquad
		\Vn:=\Wn/2\Diff^{-1}.
	\end{equation*}
	Since
	\begin{equation*}
		2\Diff^{-1}\nf^{-1}/2\Diff^{-1}
		\cong
		\Ob/\nf
	\end{equation*}
	as $\Ob$-modules, $\Vn$ is cyclic. We call $w\in\Wn$
	\emph{primitive} if its image generates $\Vn$; equivalently,
	\begin{equation*}
		w\notin2\Diff^{-1}\nf^{-1}\pf
		\qquad\text{for every prime }\pf\mid\nf.
	\end{equation*}
	For such $w$, put
	\begin{equation*}
		\psi_w(x):=e(xw/2),
		\qquad x\in\Ob.
	\end{equation*}
	%checked!!
	\begin{lemma}\label{lem:add}
		If $w \in W_{\nf}$ is primitive, then $\psi_w$ is a primitive additive character of $\kn=\Ob/\nf$.
	\end{lemma}
	\begin{proof}
		For $x\in\nf$ we have $xw/2\in\nf\cdot\Diff^{-1}\nf^{-1}=\Diff^{-1}$, hence $e(xw/2)=1$; so
		$\psi_w$ factors through $\Ob/\nf$. Let $\pf$ be a prime factor of $\nf$. Then $\psi_w$ is
		trivial on the strictly larger ideal $\nf\pf^{-1}\supsetneq\nf$ iff $\Tr(x\cdot w/2)\in\mathbb{Z}$
		for all $x\in\nf\pf^{-1}$, iff $w/2\in\Diff^{-1}\nf^{-1}\pf$. Since $w\notin 2\Diff^{-1}\nf^{-1}\pf$
		for every prime factor $\pf$ of $\nf$, $\psi_w$ does not factor through $\Ob/\nf\pf^{-1}$; hence
		it is nontrivial on every ideal strictly larger than $\nf$, i.e.\ primitive $\bmod\,\nf$.
	\end{proof}
	%checked!!
	Define the Gauss sum $\displaystyle\tau(\chi,w):=\sum_{a\in\kn}\chi(a)\,\psi_w(a)$. Since
	$\chi$ and $\psi_w$ are both primitive $\bmod\,\nf$, $|\tau(\chi,w)|^2=\Norm\nf$; in particular
	$\tau(\chi,w)\neq0$. Moreover, for every $m\in\kn$,
	\begin{equation}\label{eq:star}
		\sum_{a\in\kn}\chi(a)\,\psi_w(am)=\bar\chi(m)\,\tau(\chi,w).
	\end{equation}
	\begin{lemma}\label{lem:11}
		Assume that $\chi$ is a primitive character of $(\Ob/\nf)^{\times}$. For every $f$ as in \eqref{eq:fourier} and every primitive $w \in W_{\nf}$,
		\begin{equation*}
			(f\otimes\chi)(z)=\frac{1}{\tau(\bar\chi,w)}\sum_{r\in\Ob/\nf}\bar\chi(r)\,f(z+rw).
		\end{equation*}
	\end{lemma}
	\begin{proof}
		Each summand depends only on $r\bmod\nf$: for $r\in\nf$, $rw\in\nf\cdot2\Diff^{-1}\nf^{-1}
		=2\Diff^{-1}$, and $f$ is invariant under $z\mapsto z+2\Diff^{-1}$ (as $e(\xi\b)=1$ for
		$\xi\in2^{-1}\Ob,\ \b\in2\Diff^{-1}$). Expanding $f(z+rw)=\sum_{ \xi\in2^{-1}\Ob_+ \cup \{0\}} a_\xi e(\xi z)e(\xi rw)$,
		\begin{equation*}
			\sum_{r \in (\Ob/\nf)^{\times}}\bar\chi(r)\,f(z+rw)
			=\sum_{ \xi\in2^{-1}\Ob_+ \cup \{0\}} a_\xi\,e(\xi z)\sum_{r}\bar\chi(r)\,e(\xi rw).
		\end{equation*}
		Now $e(\xi rw)=\psi_w\big(r\cdot 2\xi\big)$;
		so with $m=2\xi$ and \eqref{eq:star} applied to $\bar\chi$,
		\begin{equation*}
			\sum_{r}\bar\chi(r)\,e(\xi rw)=\sum_{r}\bar\chi(r)\,\psi_w(rm)
			=\chi(m)\,\tau(\bar\chi,w)=\chi(2\xi)\,\tau(\bar\chi,w).
		\end{equation*}
		Hence $\sum_r\bar\chi(r)f(z+rw)=\tau(\bar\chi,w)\sum_\xi\chi(2\xi)a_\xi e(\xi z)
		=\tau(\bar\chi,w)\,(f\otimes\chi)(z)$.
	\end{proof}
	%checked!!
	\subsection{Metaplectic preliminaries}\label{ss:meta}
	Write $\G':=\Gz{4\Ob}$. For $\g\in\G'$, the factor 
	$h(\g,z)=\theta(\g z)/\theta(z)$ is defined as in 
	\S\ref{ss:background}. For Lemma~\ref{lem:phase}, we also use 
	Shimura's metaplectic setup \cite[\S3]{Sh}, specialised to $m=1$, so 
	that $G_F=\SL_2(F)$. 
	
	Put 
	\begin{equation*} 
		\iota=\begin{psmallmatrix}0&-1\\1&0\end{psmallmatrix}, 
		\qquad 
		P_{\bA}=\{g\in G_{\bA}:c_g=0\}, 
		\qquad 
		P_F=G_F\cap P_{\bA}, 
	\end{equation*} 
	and let 
	\begin{equation*} 
		\Omega_{\bA}=P_{\bA}\iota P_{\bA}. 
	\end{equation*} 
	Let 
	\begin{equation*} 
		\mathsf r_\Omega:\Omega_{\bA}\longrightarrow M_{\bA} 
	\end{equation*} 
	be the map of \cite[(3.1c), pp.~292--293]{Sh}. It is compatible 
	with the rational splitting $\mathsf r$ of \S\ref{ss:background}: 
	\begin{equation*} 
		\mathsf r=\mathsf r_\Omega 
		\qquad\text{on }P_F\iota P_F 
	\end{equation*} 
	\cite[p.~293]{Sh}.
	Let $e_v$ and $e_{\bA}$ be the additive characters of 
	\cite[\S3]{Sh}, and normalize the Haar measure on $F_v$, $v\in\bff$, so that 
	$\Ob_v$ has measure $\Norm(\Diff_v)^{-1/2}$, self-dual with respect to 
	$(x,y)\mapsto e_v(xy)$; for $v\in\ba$ we take Lebesgue measure, which is
		likewise self-dual. For $v\in\bff$, the character $e_v$ is trivial 
	on $\Diff_v^{-1}$. Put 
	\begin{equation*} 
		e_{\ba}=\prod_{v\in\ba}e_v; 
	\end{equation*} 
	on the diagonal copy of $F$, this agrees with the character $e$ of 
	\S\ref{ss:background}. We also write 
	\begin{equation*} 
		|x|_{\bA}=\prod_v|x_v|_v 
		\qquad(x\in\bA_F^\times), 
	\end{equation*} 
	so that $|c|_{\bA}=1$ for $c\in F^\times$ by the product formula. 
	
	The metaplectic group $M_{\bA}$ acts on the Schwartz--Bruhat space 
	$\mathcal S(X_{\bA})$ through the Weil representation; we write 
	$[\a f](x)$ for this action. For $\g\in\G'$, write $\hstar{\g}{z}$ for Shimura's metaplectic factor. By 
	\cite[Proposition~3.2]{Sh} and \cite[(3.10a)]{Sh}, 
	\begin{equation}\label{eq:h10a} 
		\bigl[\g^\ast\Phi_{z,0}\bigr](0) 
		= 
		\Norm\bigl(\il(\g)\bigr)^{1/2} 
		\hstar{\g}{z}^{-1}, 
	\end{equation} 
	where 
	\begin{equation*} 
		\il(\g)=c_\g\Diff^{-1}+d_\g\Ob 
	\end{equation*} 
	as in \S\ref{ss:background}. Moreover, 
	$\hstar{\g}{\cdot}$ is nowhere vanishing on $\HCal^{\ba}$. 
	
	Finally, $\Phi_{z,u}$ denotes the function of 
	\cite[(3.9a)--(3.9c), (3.6)]{Sh}. In Shimura's notation, 
	$m=1$ and $X=F^1_1=F$, so that $X_{\bA}=\bA_F$, equipped with the 
	restricted product of the local measures fixed above. We use the 
	corresponding restricted tensor-product factorization of the
		integrals over $X_{\bA}$, as in \cite[p.~302]{Sh}, with 
	$\mathbf 1_{\Ob_v}$ denoting the characteristic function of $\Ob_v$.

%checked!!

	\subsection{Phase stability}\label{ss:phase}
	Before turning to the main result of this subsection, we record a
	local lemma on squares, which is used in the
	proof of Lemma~\ref{lem:phase} and again in the proof of
	Lemma~\ref{lem:dyadic}.
	\begin{lemma}\label{lem:localsquare}
		Let $K$ be a nonarchimedean local field of characteristic $0$,
		with ring of integers $\mathcal O_K$, maximal ideal $\mathfrak p$, and
		normalized valuation $v$. Put $e=v(2)$. Then
		\begin{equation*}
			1+\mathfrak p^{\,2e+1}
			\subseteq
			\mathcal O_K^{\times2}.
		\end{equation*}
	\end{lemma}
	
	\begin{proof}
		Let $u\in1+\mathfrak p^{\,2e+1}$ and put
		$f(X)=X^2-u$. Then
		\begin{equation*}
			v(f(1))\geq2e+1>2e=2v(f'(1)).
		\end{equation*}
		Hensel's lemma therefore gives a root of $f$ in $K$.
		Since $u$ is a unit, its square root is also a unit, and hence
		$u\in\mathcal O_K^{\times2}$.
	\end{proof}
	
	The following lemma gives the identity of factors of automorphy needed for the twisting argument. Its proof passes from the classical theta factor to Shimura's metaplectic factor and uses the Weil representation to reduce the identity to
	local calculations.
	\begin{lemma}\label{lem:phase} 
		Assume $\mf\subseteq 4\Ob$ and $\nf\neq\Ob$. Let 
		$\g=\begin{psmallmatrix}a&b\\c&d\end{psmallmatrix}\in\Gz{\mf\nf^{2}}$, let $w\in\Wn$, 
		and let $r,t_r\in\Ob$ satisfy $t_r\equiv d^{2}r\pmod{\nf}$; put $\b_r=rw$, 
		$\b_{t_r}=t_r w$ and, writing
			$\ub{\b}=\begin{psmallmatrix}1&\beta\\0&1\end{psmallmatrix}$ for $\b\in F$,
		\begin{equation}\label{eq:twist} 
			\g_r:=\ub{\b_r}\,\g\,\ub{\b_{t_r}}^{-1}, 
			\qquad \ub{\b_r}\,\g=\g_r\,\ub{\b_{t_r}}. 
		\end{equation} 
		Then $\g_r\in\Gz{\mf\nf^{2}}\subseteq\G'$, so that both sides below are defined by the 
		theta quotient of \S\ref{ss:background}, and for all $z\in\HCal^{\ba}$, 
		\begin{equation*} 
			h(\g,z)=h(\g_r,\,z+\b_{t_r}). 
		\end{equation*} 
	\end{lemma}

\begin{proof} 
	Write 
	$\g_r=\begin{psmallmatrix}a_r&b_r\\ c_r&d_r\end{psmallmatrix}$. 
	Matrix multiplication gives 
	\begin{equation}\label{eq:gammar} 
		\g_r= 
		\begin{pmatrix} 
			a+\b_r c & 
			b+w\bigl(rd-(a+\b_r c)t_r\bigr)\\[2pt] 
			c & d-\b_{t_r}c 
		\end{pmatrix}, 
		\qquad 
		c_r=c,\qquad 
		d_r=d-\b_{t_r}c . 
	\end{equation} 
	Using 
	\begin{equation*} 
		w\in2\Diff^{-1}\nf^{-1}, 
		\qquad 
		c\in2^{-1}\mf\nf^2\Diff, 
		\qquad 
		t_r\equiv d^2r\pmod{\nf}, 
	\end{equation*} 
	the entries in \eqref{eq:gammar} satisfy the defining conditions of 
	$\Gz{\mf\nf^2}$. Hence 
	\begin{equation}\label{eq:ingammaprime} 
		\g,\g_r\in\Gz{\mf\nf^2} 
		\subseteq\Gz{4\Ob}=\G', 
	\end{equation} 
	the last inclusion following from $\mf\subseteq4\Ob$. 
	
	Applying \cite[(2.9c)]{Sh} with $\mathfrak c=4\Ob$, the 
	metaplectic factor of \cite[(3.10a)]{Sh} agrees on $\G'$ with 
	the theta factor. Thus 
	\begin{equation}\label{eq:metaid} 
		\hstar{\g}{z}=h(\g,z), 
		\qquad 
		\hstar{\g_r}{z+\b_{t_r}} 
		= 
		h(\g_r,z+\b_{t_r}). 
	\end{equation} 
	Therefore it suffices to prove that 
	\begin{equation*} 
		D(z):= 
		\frac{\hstar{\g_r}{z+\b_{t_r}}} 
		{\hstar{\g}{z}} 
		=1. 
	\end{equation*} 
	
	Next, we consider the case that $c=0$. 
	Then $ad=1$, so $a,d\in\Ob^\times$, and \eqref{eq:gammar} gives 
	\begin{equation*} 
		\g_r=\begin{psmallmatrix}a&b_r\\0&d\end{psmallmatrix}. 
	\end{equation*} 
	Thus $\g,\g_r\in P_F$ with 
	$d_{\g_r}=d_\g=d\in\Ob^\times$. Formula \cite[(2.9b)]{Sh} gives 
	\begin{equation*} 
		\hstar{\g}{z} 
		= 
		\hstar{\g_r}{z+\b_{t_r}} 
		= 
		\bigl|\Norm_{F/\QQ}(d)\bigr|^{1/2} 
		= 
		1, 
	\end{equation*} 
	independently of $z$. Hence $D\equiv1$. 
	
	Next we assume $c \neq 0$.  
	First, 
	\begin{equation*} 
		\il(\g)=\il(\g_r)=\Ob. 
	\end{equation*} 
	Indeed, 
	\begin{equation*} 
		c\Diff^{-1}+d\Ob\subseteq\Ob, 
		\qquad 
		1=ad-bc\in c\Diff^{-1}+d\Ob, 
	\end{equation*} 
	and the same argument applies to $\g_r$. Hence the factor 
	$\Norm(\il)^{1/2}$ in \cite[(3.10a)]{Sh} is $1$ in both cases. 
	
	Since $c_r=c\neq0$, $\g,\g_r\in P_F\iota P_F$, and the rational 
	splitting $\mathsf r$ agrees with $\mathsf r_\Omega$ on this set 
	\cite[p.~293]{Sh}. For $m=1$, evaluating \cite[(3.3)]{Sh} at 
	$x=0$ and using \cite[(3.4)]{Sh} yields 
	\begin{equation*} 
		\bigl[\mathsf r_\Omega(\g)f\bigr](0) 
		= 
		|c|_{\bA}^{1/2} 
		\int_{X_{\bA}} 
		f(yc)\, 
		e_{\bA}\!\left(\frac{cd\,y^2}{2}\right)dy. 
	\end{equation*} 
	With $t=yc$ we have $dt=|c|_{\bA}\,dy$, and 
	$|c|_{\bA}=1$ by the product formula. Thus 
	\begin{equation}\label{eq:rOmega} 
		\bigl[\mathsf r_\Omega(\g)f\bigr](0) 
		= 
		\int_{X_{\bA}} 
		f(t)\, 
		e_{\bA}\!\left(\frac{d\,t^2}{2c}\right)dt, 
	\end{equation} 
	and similarly for $\g_r$ with $d$ replaced by $d_r$. 
	
	Put 
	\begin{equation*} 
		z':=z+\b_{t_r}, 
		\qquad 
		\s:=\frac dc, 
		\qquad 
		\s_r:=\frac{d_r}{c}. 
	\end{equation*} 
	By \cite[(3.9a)--(3.9c), (3.6)]{Sh}, 
	\begin{equation*} 
		\Phi_{z,0} 
		= 
		\Phi_{\ba}(\,\cdot\,;z,0) 
		\otimes 
		\prod_{v\in\bff}\mathbf 1_{\Ob_v}, 
		\qquad 
		\Phi_{\ba}(t;z,0)=e_{\ba}(zt^2/2). 
	\end{equation*} 
	We compute, writing $I_{\infty}$ for the archimedean component of the 
	intertwining integral and $J_v$ for its component at a finite place $v$, 
	\begin{align*} 
		I_\g(z) 
		&:= 
		\bigl[\mathsf r_\Omega(\g)\Phi_{z,0}\bigr](0) 
		= 
		I_\infty(z,\s) 
		\prod_{v\in\bff}J_v(\s),\\ 
		I_{\g_r}(z') 
		&:= 
		\bigl[\mathsf r_\Omega(\g_r)\Phi_{z',0}\bigr](0) 
		= 
		I_\infty(z',\s_r) 
		\prod_{v\in\bff}J_v(\s_r), 
	\end{align*} 
	where 
	\begin{equation*} 
		J_v(\s) 
		:= 
		\int_{\Ob_v} 
		e_v\!\left(\frac{\s t^2}{2}\right)dt. 
	\end{equation*} 
	The archimedean factor depends only on $z+\s$, and 
	\eqref{eq:gammar} gives 
	\begin{equation*} 
		z'+\s_r 
		= 
		z+\b_{t_r} 
		+\frac{d-\b_{t_r}c}{c} 
		= 
		z+\s. 
	\end{equation*} 
	Hence the archimedean factors agree. Since 
	$\il(\g_r)=\Ob$, \eqref{eq:h10a} gives 
	\begin{equation*} 
		I_{\g_r}(z') 
		= 
		\hstar{\g_r}{z'}^{-1} 
		\neq0. 
	\end{equation*} 
	Thus the common archimedean factor and the finite product occurring 
	in $I_{\g_r}(z')$ are nonzero, so the common archimedean factor may 
	be cancelled. Since $\il(\g)=\il(\g_r)=\Ob$, 
	\eqref{eq:h10a} gives 
	\begin{equation}\label{eq:Dprod} 
		D(z) 
		= 
		\frac{I_\g(z)}{I_{\g_r}(z')} 
		= 
		\frac{\displaystyle\prod_{v\in\bff}J_v(\s)} 
		{\displaystyle\prod_{v\in\bff}J_v(\s_r)}. 
	\end{equation} 
	
	We compute each of the factors in \eqref{eq:Dprod}. Suppose first that $v\nmid\nf$. Since 
	$w\in2\Diff^{-1}\nf^{-1}$ and $\nf_v=\Ob_v$, we have 
	\begin{equation*} 
		\b_{t_r}=t_rw\in2\Diff_v^{-1}. 
	\end{equation*} 
	Moreover, by \eqref{eq:gammar}, 
	\begin{equation*} 
		\s-\s_r 
		= 
		\frac{d-d_r}{c} 
		= 
		\b_{t_r}. 
	\end{equation*} 
	Hence, for every $x\in\Ob_v$, 
	\begin{equation*} 
		\frac{x^2}{2}(\s-\s_r) 
		= 
		\frac{x^2\b_{t_r}}{2} 
		\in\Diff_v^{-1}. 
	\end{equation*} 
	
	Since the additive character $e_v$ is trivial on $\Diff_v^{-1}$, 
	the integrands defining $J_v(\s)$ and $J_v(\s_r)$ agree on $\Ob_v$. 
	Thus 
	\begin{equation*} 
		J_v(\s)=J_v(\s_r). 
	\end{equation*} 
	
	Now let $v\mid\nf$. Since $ad\equiv1\pmod{\pf_v}$, 
	we have $d\in\Ob_v^\times$. Put 
	\begin{equation*} 
		u_v:=d_rd^{-1} 
		= 
		1-\b_{t_r}cd^{-1}. 
	\end{equation*} 
	From the definitions of $\b_{t_r}$ and $\Gz{\mf\nf^2}$, 
	\begin{equation*} 
		\ord_v(\b_{t_r}c) 
		\geq 
		\ord_v(\mf\nf), 
	\end{equation*} 
	and hence 
	\begin{equation*} 
		u_v\in1+\mf\nf\Ob_v. 
	\end{equation*} 
	Since $v\mid\nf$ and $\mf\subseteq4\Ob$, 
	\begin{equation*} 
		\ord_v(\mf\nf) 
		\geq 
		2\ord_v(2)+1. 
	\end{equation*} 
	Therefore Lemma~\ref{lem:localsquare} gives 
	\begin{equation*} 
		u_v=\eta_v^2 
		\qquad 
		\text{for some }\eta_v\in\Ob_v^\times. 
	\end{equation*} 
	Since $\s_r=u_v\s=\eta_v^2\s$, the substitution 
	$y=\eta_vx$ preserves both $\Ob_v$ and its Haar measure, and therefore 
	\begin{equation*} 
		J_v(\s_r)=J_v(\s). 
	\end{equation*} 
	
	Thus $J_v(\s)=J_v(\s_r)$ for every finite place $v$, and 
	\eqref{eq:Dprod} gives $D\equiv1$. 
\end{proof}

%checked!!
%%%
	\subsection{The twisting theorem}\label{ss:twistthm}
	
	\begin{remark}\label{newdef}
	For the twisting argument below, we use $\Mk{\mf}{\psi}$ in a slightly
	more general classical sense. If
	\begin{equation*}
		\psi:(\Ob/\mf)^\times\longrightarrow\mathbb C^\times
	\end{equation*}
	is any character with $\psi(-1)=1$, let $\Mk{\mf}{\psi}$ denote the
	space of holomorphic functions $f:\HCal^{\ba}\to\mathbb C$, subject to
	the usual holomorphy condition at the cusps when $F=\mathbb Q$, satisfying
	\begin{equation*}
		f|_k[\g]=\psi(d_\g)f
		\qquad(\g\in\Gz{\mf}).
	\end{equation*}
	This is well defined: for $\g\in\Gz{\mf}$ one has
	$a_\g d_\g\equiv1\pmod{\mf}$, so $d_\g\in(\Ob/\mf)^\times$, and
	\begin{equation*}
		d_{\g_1\g_2}\equiv d_{\g_1}d_{\g_2}\pmod{\mf}
		\qquad(\g_1,\g_2\in\Gz{\mf}).
	\end{equation*}
	
	When $\psi$ is induced by a finite-order Hecke character, this agrees
	with Definition~\ref{def:hmf}. We use the generalized notation only in
	the classical twisting argument; by Corollary~\ref{cor:quadratic-twist},
	a quadratic twist retains the original Hecke nebentypus, so the operators
	$U$ and $V$ may subsequently be applied as in Takai.
\end{remark}
	%checked!!
	
	\begin{lemma}\label{lem:12}
		Let $f\in\Mk{\mf}{\psi}$ and let $\chi$ be a primitive residue class character on $(\Ob/\nf)^{\times}$.  Then $f\otimes\chi\in\Mk{\mf\nf^{2}}{\psi\chi^{2}}$. 
		Here $\psi$ is a character of $(\Ob/\mf)^{\times}$ with $\psi(-1)=1$,
		as in Remark~\ref{newdef}, and $\psi\chi^{2}$ is the character of
		$(\Ob/\mf\nf^{2})^{\times}$ induced by $\psi$ and $\chi^{2}$.
	\end{lemma}
	
\begin{proof}
	We regard $\psi$ and $\chi$ as characters modulo $\mf\nf^{2}$ via the
	natural reduction maps. Then $\psi\chi^{2}$ is a character of
	$(\Ob/\mf\nf^{2})^{\times}$, and
	\begin{equation*}
		(\psi\chi^{2})(-1)
		=
		\psi(-1)\chi(-1)^{2}
		=
		\psi(-1)
		=
		1.
	\end{equation*}
	
	Choose a primitive $w\in\Wn$ and put $\b_r=rw$. Fix
	\begin{equation*}
		\g=
		\begin{psmallmatrix}
			a&b\\ c&d
		\end{psmallmatrix}
		\in\Gz{\mf\nf^{2}}.
	\end{equation*}
	Since $bc\in\mf\nf^{2}\subseteq\nf$ and $ad-bc=1$, we have
	\begin{equation*}
		ad\equiv1\pmod{\nf},
	\end{equation*}
	so $d$ is invertible modulo $\nf$. For each $r\in\Ob/\nf$, choose
	$t_r\in\Ob$ satisfying
	\begin{equation*}
		t_r\equiv d^{2}r\pmod{\nf},
	\end{equation*}
	and put
	\begin{equation*}
		\g_r:=\ub{\b_r}\,\g\,\ub{\b_{t_r}}^{-1}.
	\end{equation*}
	Since multiplication by $d^{2}$ is a bijection of $\Ob/\nf$, the
	residue classes $t_r\bmod\nf$ run through $\Ob/\nf$ as $r$ does.
	
	By Lemma~\ref{lem:phase},
	\begin{equation*}
		\g_r\in\Gz{\mf\nf^{2}}\subseteq\Gz{\mf}
	\end{equation*}
	and
	\begin{equation*}
		h(\g_r,z+\b_{t_r})=h(\g,z).
	\end{equation*}
	Moreover, \eqref{eq:gammar} gives
	\begin{equation*}
		d_{\g_r}=d-\b_{t_r}c\equiv d\pmod{\mf},
	\end{equation*}
	since
	\begin{equation*}
		\b_{t_r}c
		\in
		(2\Diff^{-1}\nf^{-1})
		(2^{-1}\mf\nf^{2}\Diff)
		=
		\mf\nf
		\subseteq\mf.
	\end{equation*}
	The relation
	\begin{equation*}
		\ub{\b_r}\g=\g_r\ub{\b_{t_r}}
	\end{equation*}
	therefore gives, with $z':=z+\b_{t_r}$,
	\begin{equation}\label{eq:translate}
		\begin{aligned}
			f(\g z+\b_r)
			&=f(\g_r z')\\
			&=\psi(d_{\g_r})\,h(\g_r,z')^{k}f(z')\\
			&=\psi(d)\,h(\g,z)^{k}f(z+\b_{t_r}).
		\end{aligned}
	\end{equation}
	
	Using \eqref{eq:translate} together with Lemma~\ref{lem:11}, we obtain
\begin{equation*}
	\begin{aligned}
		(f\otimes\chi)|_k[\g](z)
		&=
		\frac{h(\g,z)^{-k}}{\tau(\bar\chi,w)}
		\sum_{r\in\Ob/\nf}
		\bar\chi(r)\,f(\g z+\b_r)\\
		&=
		\frac{\psi(d)}{\tau(\bar\chi,w)}
		\sum_{r\in\Ob/\nf}
		\bar\chi(r)\,f(z+\b_{t_r}).
	\end{aligned}
\end{equation*}
	Re-indexing by $s=t_r\bmod\nf$, we have
	\begin{equation*}
		r\equiv d^{-2}s\pmod{\nf},
	\end{equation*}
	and hence
	\begin{equation*}
		\bar\chi(r)
		=
		\bar\chi(d^{-2})\,\bar\chi(s)
		=
		\chi(d)^{2}\bar\chi(s).
	\end{equation*}
	Thus
	\begin{equation*}
		\begin{aligned}
			(f\otimes\chi)|_k[\g](z)
			&=
			\psi(d)\chi(d)^{2}
			\frac{1}{\tau(\bar\chi,w)}
			\sum_{s\in\Ob/\nf}
			\bar\chi(s)\,f(z+sw)\\
			&=
			(\psi\chi^{2})(d)\,(f\otimes\chi)(z).
		\end{aligned}
	\end{equation*}
	
	The required holomorphy follows from the finite-sum expression in
	Lemma~\ref{lem:11}; when $F=\mathbb Q$, the same expression also
	preserves holomorphy at the cusps. Therefore
	\begin{equation*}
		f\otimes\chi\in\Mk{\mf\nf^{2}}{\psi\chi^{2}}.
	\end{equation*}
\end{proof}
		
	%checked!!
	\begin{corollary}\label{cor:quadratic-twist}
		Let $\psi$ be a Hecke character whose conductor divides $\mf$, let
		$f\in\Mk{\mf}{\psi}$, and let $\chi$ be a primitive quadratic
		character of $(\Ob/\nf)^\times$. Then
		\begin{equation*}
			f\otimes\chi\in\Mk{\mf\nf^2}{\psi}.
		\end{equation*}
	\end{corollary}
	
	\begin{proof}
		Regard $\psi$ as the character of $(\Ob/\mf)^\times$ occurring in
		the transformation law. By Lemma~\ref{lem:12},
		\begin{equation*}
			f\otimes\chi\in\Mk{\mf\nf^2}{\psi\chi^2}.
		\end{equation*}
		Since $\chi$ is quadratic, $\chi^2$ is the trivial character on
		$(\Ob/\nf)^\times$. Hence the nebentypus of $f\otimes\chi$ is
		$\psi$. Since $\psi$ is a Hecke character whose conductor divides
		$\mf\mid\mf\nf^2$, the twisted form lies in
		$\Mk{\mf\nf^2}{\psi}$ in the original sense.
	\end{proof}
	%check!!
	\subsection{A dyadic local quadratic character detecting ramification}
%checked

	\begin{lemma}\label{lem:dyadic}
		Let $F$ be a number field and let $\qf\mid2$ be a prime ideal of $\Ob$.
		Put $K=F_\qf$, let $\Ob_{\qf}$ be the valuation ring of $K$, and let
		$v_\qf$ be the normalized valuation on $K$. Set $e=v_\qf(2)$.
		Then there exist an integer $n$ with $2\leq n\leq2e$, a primitive
		quadratic character
		\begin{equation*}
			\chi_\qf:(\Ob/\qf^{\,n})^\times
			\longrightarrow\{\pm1\},
		\end{equation*}
		extended by $0$ to $\Ob/\qf^{\,n}$, and a sign
		$\varepsilon\in\{\pm1\}$ such that, for every
		$\xi\in\tfrac12\Ob$,
		\begin{equation*}
			\chi_\qf(2\xi)=\varepsilon
			\quad\Longrightarrow\quad
			\qf\text{ ramifies in }F(\sqrt{-2\xi})/F.
		\end{equation*}
		Moreover, if $a_0\in\Ob$ is a $\qf$-adic unit and $\qf$ ramifies in
		$F(\sqrt{-a_0})/F$, then $\chi_\qf$ and $\varepsilon$ may be chosen so that
		\begin{equation*}
			\chi_\qf(a_0)=\varepsilon.
		\end{equation*}
	\end{lemma}
	%checked
\begin{proof}
	Put
	\begin{equation*}
		U=\Ob_{\qf}^\times/\Ob_{\qf}^{\times2}.
	\end{equation*}
	By \cite[Chapter~II, \S5, Corollary~5.8]{Neukirch},
	\begin{equation*}
		\dim_{\mathbb F_2}K^\times/K^{\times2}
		=[K:\mathbb Q_2]+2.
	\end{equation*}
	Since
	\begin{equation*}
		K^\times/K^{\times2}
		\cong(\mathbb Z/2\mathbb Z)\oplus U,
	\end{equation*}
	it follows that
	\begin{equation*}
		\dim_{\mathbb F_2}U=[K:\mathbb Q_2]+1\geq2.
	\end{equation*}
	
	By \cite[Chapter~III, \S5, Theorem~2]{SerreLocalFields},
	finite unramified extensions of $K$ correspond to finite separable
	extensions of its residue field. Since the residue field is finite,
	it has a unique quadratic extension, and hence $K$ has a unique
	unramified quadratic extension. The corresponding square class has
	even valuation, so we may choose a representative
	$u_0\in\Ob_{\qf}^\times$. Thus, for $a\in K^\times$, the extension
	$K(\sqrt a)/K$ is trivial or unramified if and only if
	\begin{equation*}
		[a]\in\{[1],[u_0]\}.
	\end{equation*}
	
	Since $[u_0]\neq[1]$ and $\dim_{\mathbb F_2}U\geq2$, there exists a
	nonzero $\mathbb F_2$-linear functional
	\begin{equation*}
		\lambda:U\longrightarrow\mathbb F_2
	\end{equation*}
	such that $\lambda([u_0])=0$. If $a_0$ is as in the final assertion
	of the lemma, then
	\begin{equation*}
		[-a_0]\notin\{[1],[u_0]\},
	\end{equation*}
	since $\qf$ ramifies in $F(\sqrt{-a_0})/F$. Hence
	$[-a_0]\notin\langle[u_0]\rangle$, and we may choose $\lambda$ so
	that, in addition,
	\begin{equation*}
		\lambda([-a_0])=1.
	\end{equation*}
	Define
	\begin{equation*}
		\widetilde\chi(a)=(-1)^{\lambda([a])},
		\qquad a\in\Ob_{\qf}^\times.
	\end{equation*}
	Then $\widetilde\chi$ is a nontrivial quadratic character and
	$\widetilde\chi(u_0)=1$.
	
	We next show that $\widetilde\chi$ has finite conductor. If
	$u\in1+\qf^{\,2e+1}\Ob_{\qf}$, then Lemma~\ref{lem:localsquare}
	gives $u\in\Ob_{\qf}^{\times2}$. Therefore
	\begin{equation*}
		1+\qf^{\,2e+1}\Ob_{\qf}
		\subseteq
		\Ob_{\qf}^{\times2}
		\subseteq
		\ker\widetilde\chi.
	\end{equation*}
	In fact $\widetilde\chi$ is trivial on
		$1+\qf^{\,2e}\Ob_{\qf}$. Indeed, every element of this group has the
		form $1+4y$ with $y\in\Ob_{\qf}$, and $(1+2x)^2=1+4(x+x^2)$ for
		$x\in\Ob_{\qf}$. If the residue class of $y$ has the form
		$\bar x^2+\bar x$, then
		$(1+4y)(1+2x)^{-2}\in1+\qf^{\,2e+1}\Ob_{\qf}$, so that $1+4y$ is a
		square. Otherwise the polynomial $T^2-T-y$ is irreducible modulo
		$\qf$, so that its root $(1+\sqrt{1+4y})/2$ generates the unramified
		quadratic extension of $K$, and $[1+4y]=[u_0]$. As
		$\widetilde\chi(u_0)=1$, we obtain
		\begin{equation*}
			1+\qf^{\,2e}\Ob_{\qf}
			\subseteq
			\ker\widetilde\chi.
	\end{equation*}
	
	Let $\qf^{\,n}$ be the conductor of $\widetilde\chi$, so that $n$ is
	the smallest positive integer for which
	\begin{equation*}
		1+\qf^{\,n}\Ob_{\qf}
		\subseteq\ker\widetilde\chi.
	\end{equation*}
	The preceding inclusion gives $n\leq2e$. Moreover,
	$(\Ob/\qf)^\times$ has odd order, since the residue field has
	cardinality $2^f$ for some $f\geq1$. Hence it admits no nontrivial
	quadratic character. Since $\widetilde\chi$ is nontrivial, it cannot
	factor modulo $\qf$, and therefore $n\geq2$.
	%checked!!
	Since
	$1+\qf^{\,n}\Ob_{\qf}\subseteq\ker\widetilde\chi$,
	$\widetilde\chi$ is constant on residue classes modulo $\qf^{\,n}$
	and therefore descends to a character on
	\begin{equation*}
		\Ob_{\qf}^\times/(1+\qf^{\,n}\Ob_{\qf})
		\cong
		(\Ob_{\qf}/\qf^{\,n}\Ob_{\qf})^\times
		\cong
		(\Ob/\qf^{\,n})^\times.
	\end{equation*}
	Denote the resulting character by
	\begin{equation*}
		\chi_\qf:(\Ob/\qf^{\,n})^\times
		\longrightarrow\{\pm1\}.
	\end{equation*}
	By the minimality of $n$, $\chi_\qf$ is primitive. Extend it by $0$
	to all of $\Ob/\qf^{\,n}$.
	
	Finally, put
	\begin{equation*}
		\varepsilon=-\widetilde\chi(-1).
	\end{equation*}
	Let $\xi\in\tfrac12\Ob$, put $a=2\xi$, and suppose
	$\chi_\qf(a)=\varepsilon$. Since $\varepsilon\neq0$, the element $a$
	is a unit at $\qf$, and hence
	\begin{equation*}
		\widetilde\chi(a)
		=
		\chi_\qf(a)
		=
		-\widetilde\chi(-1).
	\end{equation*}
	Since $\widetilde\chi(u_0)=1$, we also have
	\begin{equation*}
		\widetilde\chi(-u_0)=\widetilde\chi(-1).
	\end{equation*}
	Thus
	\begin{equation*}
		[a]\notin\{[-1],[-u_0]\},
	\end{equation*}
	and therefore
	\begin{equation*}
		[-a]\notin\{[1],[u_0]\}.
	\end{equation*}
	Hence $K(\sqrt{-a})/K$ is neither trivial nor unramified, and is
	therefore ramified. Equivalently, $\qf$ ramifies in
	$F(\sqrt{-a})/F$. Since $a=2\xi$, the first assertion follows.
	
	If $a_0$ is as in the final assertion, then our choice of $\lambda$
	gives
	\begin{equation*}
		-1
		=
		\widetilde\chi(-a_0)
		=
		\widetilde\chi(-1)\widetilde\chi(a_0).
	\end{equation*}
	Therefore
	\begin{equation*}
		\chi_\qf(a_0)
		=
		\widetilde\chi(a_0)
		=
		-\widetilde\chi(-1)
		=
		\varepsilon,
	\end{equation*}
	which proves the final assertion.
\end{proof}

	%checked!!
	\subsection{Takai's result}
	We first recall Takai's theorem and then give the variant needed below.
	%checked
	\begin{theorem}[Takai]\label{thm:takai1}
			Let $F$ be a totally real number field which is Galois over $\mathbb{Q},$ let $g=[F:\mathbb{Q}]$, let $D(F)$ be the discriminant of $F/\mathbb{Q}$, and let $p$ be a
		prime such that
		\begin{equation*}
			g\le [F(\zeta_p):F]\,2^{-\ord_2([F(\zeta_p):F])}
			\qquad\text{and}\qquad p>2g+1 .
		\end{equation*}
		Let $r$ be a positive integer, let $\varepsilon_1,\varepsilon_2,\dots,\varepsilon_r\in\{0,\pm1\}$ be
		such that $\varepsilon_i\neq0$ for some $i$, and let $\chi_1,\chi_2,\dots,\chi_r$ be quadratic
		Hecke characters of $F$ whose conductors are the integral ideals $N_1,N_2,\dots,N_r$
		respectively. Set the positive integer $N$ by $N\mathbb{Z}=N_1N_2\cdots N_r\cap\mathbb{Z}$, and put
		\begin{equation*}
			A=\frac{gN^2D(F)}{8}\prod_{\substack{d\mid ND(F)\\ d\ \mathrm{prime}}}\left(1+\frac1d\right).
		\end{equation*}
		If there is a prime number $q>(A/g)^g$ such that $q$ is unramified at $F/\mathbb{Q}$ and
		\begin{equation*}
			\sum_{\substack{\xi\in 2^{-1}\Ob_{+},\ \chi_i(2\xi)=\varepsilon_i,\ i=1,2,\dots,r\\[1pt]
					\Tr(\xi)=qg/2,\ (q\Ob,\,2\xi \Ob)\neq1}}
			\beta(\xi)\,\frac{2^{g}\,h^{-}\!\big(F(\sqrt{-2\xi})\big)}
			{Q_F(\sqrt{-2\xi})\,w_F(\sqrt{-2\xi})}\not\equiv 0\pmod p,
		\end{equation*}
		then
		\begin{equation*}
			\#\left\{\,K=F(\sqrt{-2\xi})\ \middle|\
			\begin{aligned}
				&p\nmid h^{-}(K/F),\ |\Norm_{F/\mathbb{Q}}(D(K/F))|<X,\\
				&\chi_i(2\xi)=\varepsilon_i\ \text{for}\ i=1,2,\dots,r
			\end{aligned}\right\}
			\ \gg_{F,p}\ \frac{X^{1/(2g)}}{\log X}.
		\end{equation*}
		Here $h^{-}(F(\sqrt{-2\xi}))$ is the relative class number of $F(\sqrt{-2\xi})$,
		$Q_F(\sqrt{-2\xi})$ is its Hasse index, $w_F(\sqrt{-2\xi})$ is the number of
		roots of unity in $F(\sqrt{-2\xi})$, and
		\begin{equation*}
			\beta(\xi)=\sum_{\fa,\fb}\mu(\fa)\left(\frac{F(\sqrt{-2\xi})/F}{\fa}\right)\Norm_{F/\mathbb{Q}}(\fb),
		\end{equation*}
		where the pair $(\fa,\fb)$ runs over all integral ideals prime to $2\Ob$ such that
		$(\fa\fb)^2\mid 2\xi \Ob$, and $\mu$ is the M\"obius function.
	\end{theorem}

%checked!!
\begin{remark}
	In \cite[Theorem~1]{Takai}, the condition in the conclusion is printed as
	$\chi_i(\xi)=\varepsilon_i$. This should read
	$\chi_i(2\xi)=\varepsilon_i$. Indeed, with the indexing convention
	\eqref{eq:fourier}, twisting multiplies the Fourier coefficient indexed by
	$\xi$ by $\chi_i(2\xi)$; this is also the condition used in Takai's proof.
\end{remark}
	%checked!!

	%checked
	
	We first note that the two numerical hypotheses on $p$ in
		Theorem~\ref{thm:takai1} can be weakened; this is what allows $\ell=5$ over
		real quadratic fields in Theorem~\ref{thm:quadratic-version}.
	\begin{lemma}\label{lem:takai-sharp}
		Suppose $g\geq2$. The two hypotheses
		$p>2g+1$ and $g\leq[F(\zeta_p):F]\,2^{-\ord_2([F(\zeta_p):F])}$ in Theorem~\ref{thm:takai1} may be
		replaced by the single, weaker hypothesis $[F(\zeta_p):F]\geq2g$.
	\end{lemma}

	\begin{proof}
		In the proof of \cite[Theorem~1]{Takai}, the two original numerical
		hypotheses are used only to ensure that $\zeta_p$ lies in no quadratic
		extension of $F$ and that
		\begin{equation*}
			[L(\zeta_p):L]\geq g
		\end{equation*}
		for the multiquadratic extension $L/F$ occurring in the proof of
		\cite[Lemma~3]{Takai}; these are the inputs used in
		\cite[Remark~1]{Takai} and \cite[Lemma~3]{Takai}, respectively.
		
		The first follows immediately from the present hypothesis. Indeed, if
		$\zeta_p\in K$ with $[K:F]=2$, then
		\begin{equation*}
			[F(\zeta_p):F]\leq 2<2g,
		\end{equation*}
		a contradiction. In particular, $p$ is odd, so
		$(x+2)^g-x^g$ is a nonzero polynomial modulo $p$ of degree at most
		$g-1$.
		
		For the second, put
		\begin{equation*}
			M=L\cap F(\zeta_p).
		\end{equation*}
		Since $F(\zeta_p)/F$ is cyclic, $M/F$ is cyclic; and since $L/F$ is
		multiquadratic, $\Gal(M/F)$ has exponent dividing $2$. Hence
		$[M:F]\leq2$, and therefore
		\begin{equation*}
			[L(\zeta_p):L]
			=[F(\zeta_p):M]
			\geq \frac{1}{2}[F(\zeta_p):F]
			\geq g.
		\end{equation*}
		Thus the cyclotomic choices in the proof of \cite[Lemma~3]{Takai}
		contain at least $g$ elements, while $(x+2)^g-x^g$ has at most
		$g-1$ roots modulo $p$. Hence Takai's argument applies unchanged.
	\end{proof}
	
		Theorem~\ref{thm:takai1} is stated in \cite{Takai} for quadratic Hecke
	characters. In our application, the local conditions in Morrow's theorem
	are naturally encoded by characters of $(\Ob/\nf)^{\times}$ which need not
	extend to Hecke characters of $F$. We show that Takai's twisting argument
	remains valid for primitive quadratic residue class characters. The key
	input is Lemma~\ref{lem:12}, whose proof uses Lemma~\ref{lem:phase} to
	control the factor of automorphy under the translations occurring in the
	twist.
	
	\begin{theorem}\label{thm:takai2}
		Suppose $g\geq2$. Let the notation and hypotheses be as in Theorem~\ref{thm:takai1},
		except that
		\[
		[F(\zeta_p):F]\geq 2g
		\]
		replaces the two numerical hypotheses on $p$, and
		$\chi_1,\dots,\chi_r$ are primitive quadratic characters modulo the
		integral ideals $N_1,\dots,N_r$ in the sense of \S\ref{ss:twist}.
		Then the conclusion of Theorem~\ref{thm:takai1} holds.
	\end{theorem}

	\begin{proof}
		By Lemma~\ref{lem:takai-sharp}, Takai's proof of
		\cite[Theorem~1]{Takai} remains valid under the stated numerical
		hypothesis. It therefore remains only to modify the twisting step,
		which in \cite[\S2.2]{Takai} is formulated for quadratic Hecke
		characters. By Lemma~\ref{lem:12} and
		Corollary~\ref{cor:quadratic-twist}, the twists of
		Definition~\ref{def:twist} have the required modularity and, since
		the $\chi_i$ are quadratic, retain the original Hecke nebentypus.
		Their action on Fourier coefficients agrees with that of Takai's
		twisting operator. Hence the remainder of Takai's proof applies
		unchanged.
	\end{proof}
	%checked!!
	\section{Elliptic curves}\label{sec:selmer}
	\subsection{Background}\label{ss:ec-background}

		Throughout this subsection $F$ is a number field with ring of integers
	$\Ob$, and $\ell$ is an odd prime. We use the standard notation and
	conventions of \cite{SilvermanAEC}.
	
	An elliptic curve $E/F$ may be given by a Weierstrass equation
	\begin{equation*}
		E:\quad y^{2}+a_{1}xy+a_{3}y
		=
		x^{3}+a_{2}x^{2}+a_{4}x+a_{6},
		\qquad a_i\in F,
	\end{equation*}
	with discriminant $\Delta$ and $j$-invariant $j_E=c_4^3/\Delta$.
	For a prime $\pf$ of $\Ob$, let $\ord_{\pf}(\Delta)$ denote the
	discriminant exponent of a minimal Weierstrass equation at $\pf$, and write
	\begin{equation*}
		\Delta_E
		=
		\prod_{\pf}\pf^{\,\ord_{\pf}(\Delta)}
	\end{equation*}
	for the minimal discriminant ideal and $N(E)$ for the conductor of $E$.%checked
				 
				 Following \cite[\S2]{Morrow}, when $\ord_{\pf}(j_E)\geq0$ the kernel of
			reduction modulo $\pf$ is taken after passing to an extension over which
			$E$ has good reduction. When $\ord_{\pf}(j_E)<0$, after an extension of
			degree at most $2$ the curve admits a Tate parametrisation. In this case,
			a point $\tau(u)$ with $u$ a unit lies in the kernel precisely when
			$u-1$ lies in the maximal ideal
			\cite[Definition~2.5]{Morrow}. At a prime of good reduction, the condition
			that $P\in E(F)$ lie outside the kernel is simply that $P$ does not reduce
			to $O$.
	%checked
	We shall also use that $\ord_{\pf}(j_E)<0$ if and only if $E$ has
	potentially multiplicative reduction at $\pf$. In this case $E/F_{\pf}$ is
	a quadratic twist of a Tate curve $E_q$ with
	\begin{equation*}
		\ord_{\pf}(q)=-\ord_{\pf}(j_E),
	\end{equation*}
	and $E/F_{\pf}$ is itself a Tate curve if and only if it has split
	multiplicative reduction \cite[Theorem~C.14.1]{SilvermanAEC}. For a positive integer $n$, let 
	\begin{equation*} 
		S(n)=\left\{\ell\ \text{prime}:\  
		\ell\mid \#E(L)_{\tors}\ \text{for some }[L:\QQ]=n 
		\text{ and some elliptic curve }E/L 
		\right\}. 
	\end{equation*}  
	
	%checked
	% 
	For $d\in F^{\times},$ let $E^{d}$ denote the quadratic twist of $E$ by $d$. 
	If 
	\begin{equation*} 
		E:\quad y^{2}=x^{3}+Ax+B, 
	\end{equation*} 
	then 
	\begin{equation*} 
		E^{d}:\quad y^{2}=x^{3}+Ad^{2}x+Bd^{3}. 
	\end{equation*} 
	Up to $F$-isomorphism, $E^{d}$ depends only on the class of $d$ in 
	$F^{\times}/(F^{\times})^{2}$, and $E^{d}\cong E$ over $F(\sqrt d)$ 
	\cite[X.5.4]{SilvermanAEC}. 
	% 
	%checked
	% 
		Write
		\begin{equation*}
			\mathcal G_F=\Gal(\overline F/F),
			\qquad
			\mathcal G_{F_v}=\Gal(\overline{F_v}/F_v).
		\end{equation*}
		The $\ell$-Selmer group of $E/F$ is
		\begin{equation*}
			\Sel_{\ell}(E,F)
			=
			\ker\Bigl(
			H^{1}(\mathcal G_F,E[\ell])
			\longrightarrow
			\prod_v H^{1}(\mathcal G_{F_v},E)
			\Bigr),
		\end{equation*}
		where $v$ runs over all places of $F$. The Kummer sequence yields
		\begin{equation*}
			1\longrightarrow E(F)/\ell E(F)
			\longrightarrow \Sel_{\ell}(E,F)
			\longrightarrow \Sha(E/F)[\ell]
			\longrightarrow1.
		\end{equation*}
	Consequently, $\Sel_{\ell}(E,F)=1$ implies
	\begin{equation*}
		\rank E(F)=0,\qquad
		E(F)[\ell]=0,\qquad
		\Sha(E/F)[\ell]=1.
	\end{equation*}
	%checked!!

	\subsection{The hypothesis $\widetilde S_{E}=\varnothing$}\label{ss:redtype}
	
	Morrow's theorem \cite[Corollary~E]{Morrow} involves an auxiliary
	primitive Hecke character $\chi_H$ of $F$ of order $\ell$, through the sets
	\begin{equation*}
		\begin{aligned}
			\widetilde S_E(\chi_H)
			&=
			\left\{
			\pf\mid N(E):
			\chi_H(\pf)\neq0,\ 
			\ord_\pf(\Delta_E)\not\equiv0\pmod{\ell}
			\right\},
			\\
			S_E(\chi_H)
			&=
			\left\{
			\pf\in\widetilde S_E(\chi_H):
			\ord_\pf(j_E)<0
			\right\}
		\end{aligned}
	\end{equation*}
	and the hypothesis $\widetilde S_E(\chi_H)=\varnothing$. We replace this
	hypothesis by an elementary criterion involving only $\ell$,
	$\Norm\pf=\#k(\pf)$ and $\ord_\pf(\Delta_E)$ at the primes of bad reduction
	(Proposition~\ref{prop:star}); this is the form used in
	Theorem~\ref{thm:morrow2} and in \S\ref{sec:examples}. We retain the notation
	and definitions of \cite{Morrow}.

	By class field theory, $\chi_{H}$ cuts out a cyclic extension $N/F$ of degree $\ell$, and for every prime $\pf$ of $\Ob$,
	\begin{equation}\label{eq:chi-ram}
		\chi_{H}(\pf)=0
		\iff
		\pf\mid\mathfrak f(\chi_{H})
		\iff
		\pf\ \text{ramifies in}\ N/F .
	\end{equation}
	Thus $\widetilde S_{E}(\chi_{H})=\varnothing$ if and only if every
	$\pf\mid N(E)$ with $\ell\nmid\ord_{\pf}(\Delta_{E})$ ramifies in $N/F$.
	The problem therefore reduces to determining at which primes a cyclic
	extension of degree $\ell$ can ramify.
	
	%checked!!
	
	\begin{lemma}\label{lem:localchar}
		Let $\pf$ be a prime of $\Ob$ and $\ell$ an odd prime. There exists a continuous character
		\begin{equation*}
			\chi_{\pf}:F_{\pf}^{\times}\to\mu_{\ell}
		\end{equation*}
		nontrivial on $\Ob_{\pf}^{\times}$ if and only if
		\begin{equation*}
			\pf\mid\ell
			\qquad\text{or}\qquad
			\ell\mid(\Norm\pf-1).
		\end{equation*}
		
	\end{lemma}
	\begin{proof}
		Let $p$ be the residue characteristic of $\pf$ and put
		$U_{1}=1+\pf\Ob_{\pf}$. Recall that
		\begin{equation*}
			\Ob_{\pf}^{\times}\simeq\mu_{\Norm\pf-1}\times U_{1}.
		\end{equation*}
		If $\pf\nmid\ell$, then $p\neq\ell$, and every continuous character
		$U_{1}\to\mu_{\ell}$ is trivial since $U_{1}$ is pro-$p$. Thus a nontrivial
		character $\Ob_{\pf}^{\times}\to\mu_{\ell}$ exists if and only if
		$\ell\mid\Norm\pf-1$.

		If $\pf\mid\ell$, then
		\begin{equation*}
			U_{1}/(1+\pf^{2}\Ob_{\pf})\simeq k(\pf)^{+},
		\end{equation*}
		so $U_{1}$ admits a nontrivial continuous character into $\mu_{\ell}$.
		Finally, since
		$F_{\pf}^{\times}\simeq\langle\pi\rangle\times\Ob_{\pf}^{\times}$ for a
		uniformizer $\pi$, every character above extends to $F_{\pf}^{\times}$ by
		sending $\pi$ to $1$.
		\end{proof}
	
	%checked!!

	\begin{proposition}\label{prop:star}
		Let $\ell$ be an odd prime and $E/F$ an elliptic curve. Put
		$B=\{\pf\mid N(E):\ell\nmid\ord_{\pf}(\Delta_{E})\}$. The following are
		equivalent.
		\begin{enumerate}[label=\textup{(\roman*)},leftmargin=*]
			\item There is a primitive Hecke character $\chi_{H}$ of $F$ of order
			$\ell$ with $\widetilde S_{E}(\chi_{H})=\varnothing$.
			\item Every $\pf\in B$ with $\pf\nmid\ell$ satisfies
			$\ell\mid\Norm\pf-1$.
		\end{enumerate}
	\end{proposition}
	
	\begin{proof}
		Assume \textup{(ii)}. By Lemma~\ref{lem:localchar}, for each $\pf\in B$
		there is a continuous character
		$\chi_{\pf}:F_{\pf}^{\times}\to\mu_{\ell}$ of order $\ell$ which is
		nontrivial on $\Ob_{\pf}^{\times}$. Since $\ell$ is odd, the exceptional
		case of the Grunwald--Wang theorem does not occur
		\cite[Chapter~X]{ArtinTate}; hence there is a global Hecke character
		$\chi_H$ of order $\ell$ whose local component at each $\pf\in B$ is
		$\chi_{\pf}$.Writing the conductor of $\chi_H$ as
		$\mathfrak f(\chi_H)=\prod_{\pf}\pf^{a_{\pf}}$, where $a_{\pf}$ is the
		local conductor exponent at $\pf$, the nontriviality of $\chi_{\pf}$ on
		$\Ob_{\pf}^{\times}$ gives $a_{\pf}>0$ for every $\pf\in B$. Thus
		$\chi_H(\pf)=0$ for every $\pf\in B$, and hence
		$\widetilde S_E(\chi_H)=\varnothing$.

		Conversely, suppose \textup{(ii)} fails, and choose $\pf\in B$ with
		$\pf\nmid\ell$ and $\ell\nmid\Norm\pf-1$. By
		Lemma~\ref{lem:localchar}, the local component at $\pf$ of any Hecke
		character of order $\ell$ is trivial on $\Ob_{\pf}^{\times}$. Thus
		$\chi_{H}$ is unramified at $\pf$, so $\chi_{H}(\pf)\neq0$ by
		\eqref{eq:chi-ram}. Hence $\pf\in\widetilde S_{E}(\chi_{H})$, and
		\textup{(i)} fails.
		\end{proof}
	
	%checked!!
	When $E$ has an $F$-rational point of order $\ell$, condition
	\textup{(ii)} has the following interpretation in terms of the reduction of $E$.
	
	\begin{lemma}\label{lem:redtype}
		Let $\ell\geq5$ be a prime and let $E/F$ be an elliptic curve with a point
		$P\in E(F)$ of order $\ell$. Let $\pf\mid N(E)$ be a prime with
		$\pf\nmid\ell$. Then $E$ has multiplicative reduction at $\pf$. If moreover
		$\ell\nmid\ord_{\pf}(\Delta_{E})$, then
		\begin{equation*}
			\begin{aligned}
				E/F_{\pf}\ \text{is split multiplicative}
				&\iff \ell\mid\Norm\pf-1,\\
				E/F_{\pf}\ \text{is nonsplit multiplicative}
				&\iff \ell\mid\Norm\pf+1 .
			\end{aligned}
		\end{equation*}
	\end{lemma}
	
	\begin{proof}
		Let $p$ be the residue characteristic of $\pf$, so $p\neq\ell$. Write
		$
		E_{1}(F_{\pf})\subseteq E_{0}(F_{\pf})\subseteq E(F_{\pf}), \Phi_{\pf}=E(F_{\pf})/E_{0}(F_{\pf}),
		$
		with the notation of \cite[VII.2]{SilvermanAEC}. By
		\cite[VII.2.1--2, IV.3.2]{SilvermanAEC},
		$E_{1}(F_{\pf})$ has no $\ell$-torsion and
		\begin{equation*}
			E_{0}(F_{\pf})/E_{1}(F_{\pf})
			\simeq\widetilde E_{\mathrm{ns}}(k(\pf)),
		\end{equation*}
		whose order is $\Norm\pf$, $\Norm\pf-1$, or $\Norm\pf+1$ according as
		the reduction is additive, split multiplicative, or nonsplit
		multiplicative. If the reduction were additive, then
		$\#\Phi_{\pf}\leq4$ \cite[VII.6.1]{SilvermanAEC}, while
		$\#(E_{0}(F_{\pf})/E_{1}(F_{\pf}))=\Norm\pf$ is prime to $\ell$.
		Thus every $\ell$-torsion point of $E(F_{\pf})$ lies in
		$E_{1}(F_{\pf})$, a contradiction. 
		
		Assume $\ell\nmid\ord_{\pf}(\Delta_{E})$. In the split case
		$\#\Phi_{\pf}=\ord_{\pf}(\Delta_{E})$ \cite[VII.6.1]{SilvermanAEC},
		while in the nonsplit case $\#\Phi_{\pf}\mid2$
		\cite[Appendix~C, Corollary~15.2.1]{SilvermanAEC}. Hence
		$\ell\nmid\#\Phi_{\pf}$, so $P\in E_{0}(F_{\pf})$. Since
		$E_{1}(F_{\pf})$ has no $\ell$-torsion, the image of $P$ in
		$\widetilde E_{\mathrm{ns}}(k(\pf))$ has order $\ell$. Therefore
		$\ell\mid\Norm\pf-1$ in the split case and
		$\ell\mid\Norm\pf+1$ in the nonsplit case. Since $\ell$ is odd, these
		two divisibilities are mutually exclusive, giving the stated
		equivalences.
	\end{proof}
	%checked!!
	\begin{corollary}\label{cor:SE-equal}
		Let $\ell\geq5$ and suppose that $E(F)$ contains a point of order $\ell$.
		\begin{enumerate}[label=\textup{(\alph*)},leftmargin=*]
			\item If $\pf\mid N(E)$ and $\pf\nmid\ell$, then
			$\ord_{\pf}(j_{E})=-\ord_{\pf}(\Delta_{E})<0$. Hence
			$S_{E}\supseteq\{\pf\in\widetilde S_{E}:\pf\nmid\ell\}$, and the branch
			$\ord_{\pf}(j_{E})\geq0$ in condition \textup{(3)} of
			\cite[Theorem~B]{Morrow} is vacuous for such $\pf.$

			\item If $\qf\nmid\ell$ and $\ell\nmid\Norm\qf^{2}-1$, then
			$\ell\mid\ord_{\qf}(\Delta_{E})$. In particular, this holds for
			$\qf\mid2$ with $\Norm\qf\in\{2,4\}$ and $\ell\geq7$.
		\end{enumerate}
	\end{corollary}
	
	\begin{proof}
		(a) By Lemma~\ref{lem:redtype}, $E$ has multiplicative reduction at $\pf$, so
		$\ord_{\pf}(j_{E})=-\ord_{\pf}(\Delta_{E})<0$
		\cite[VII.5.1]{SilvermanAEC}. The remaining assertions follow immediately.
		
		(b) If $\qf\mid N(E)$ and $\ell\nmid\ord_{\qf}(\Delta_{E})$, then
		Lemma~\ref{lem:redtype} gives $\ell\mid\Norm\qf-1$ or
		$\ell\mid\Norm\qf+1$, contrary to
		$\ell\nmid\Norm\qf^{2}-1$. If $\qf\nmid N(E)$, then
		$\ord_{\qf}(\Delta_{E})=0$. Finally,
		$\Norm\qf^{2}-1\in\{3,15\}$ when $\Norm\qf\in\{2,4\}$.
	\end{proof}
	
	%checked!!
\subsection{Ramification at the dyadic primes}\label{ss:dyadic-ram}
	
This subsection removes two hypotheses from \cite[Corollary~E]{Morrow}: that
$\Norm_{F/\QQ}(\qf)=2$ for every $\qf\mid2$, and that $[F:\QQ]\leq5$.
Both enter the proof of \cite[Theorem~B]{Morrow} through
\cite[Lemma~5.5]{Morrow}. Lemma~\ref{lem:dyadic-ram} gives the
same conclusion without either hypothesis.
	
		\begin{lemma}\label{lem:dyadic-ram}
		Let $\ell$ be an odd prime, let $d\in F^{\times}/(F^{\times})^{2}$, and let $\qf\mid2$ be a prime of $\Ob$ which ramifies in
		$F(\sqrt d)/F$. Put $L=F(\sqrt d,\z_{\ell})$ and
		$m=[F(\z_{\ell}):F]$.
		\begin{enumerate}[label=\textup{(\alph*)},leftmargin=*]
			\item Let $M/F(\sqrt d)$ be cyclic of degree $\ell$ with $M/F$ Galois and
			$\Gal(M/F)$ non-abelian of order $2\ell$. Then $M/F(\sqrt d)$ is
			unramified at every divisor of $\qf$.
			
			\item Suppose that $m$ is even. Let $\d$ and $\s$ be generators of $\Gal(L/F(\z_{\ell}))$ and
			$\Gal(L/F(\sqrt d))$ respectively, and put $K'=L^{\langle\d\s^{m/2}\rangle}$.
			Let $M/K'$ be cyclic of degree $\ell$ with $M/F$ Galois, and suppose that the
			action of $\Gal(K'/F)$ on $\Gal(M/K')$ by conjugation is faithful. Then
			$M/K'$ is unramified at every divisor of $\qf$.
		\end{enumerate}
	\end{lemma}
	
	\begin{proof}
		In either case, let $I$ be the inertia group of a prime of $M$ above $\qf$, and let
		$I_{1}$ be its wild inertia subgroup. Since the residue characteristic is $2$,
		$I_{1}$ is a normal Sylow $2$-subgroup of $I$ and $I/I_{1}$ is cyclic
		\cite[IV, \S2]{SerreLocalFields}.
		
		(a) Since $\Gal(M/F)$ is non-abelian of order $2\ell$, it is dihedral,
		and none of its subgroups of order $2$ is normal. If a divisor of $\qf$
		in $F(\sqrt d)$ ramified in $M$, then, as $\qf$ ramifies in
		$F(\sqrt d)/F$, the ramification index above $\qf$ in $M/F$ would be
		$2\ell=[M:F]$. Hence $I=\Gal(M/F)$ and $I_{1}$ would be a normal
		subgroup of order $2$ of $\Gal(M/F)$, a contradiction.
		
		(b) Since $\qf\nmid\ell$, the extension $F(\z_{\ell})/F$ is unramified at
		$\qf$, whereas $F(\sqrt d)/F$ is ramified. Hence $F(\sqrt d)\cap F(\z_{\ell})=F$,
		so $\Gal(L/F)=\langle\d\rangle\times\langle\s\rangle$ with $\d$ of order $2$,
		and the inertia subgroup at $\qf$ has trivial image in
		$\Gal(F(\z_{\ell})/F)$ and nontrivial image in
		$\Gal(F(\sqrt d)/F)$, hence is $\langle\d\rangle$. As $\s^{m/2}\neq1$, we have
		$\langle\d\rangle\cap\langle\d\s^{m/2}\rangle=1$.
		Hence the image of the inertia group $\langle\d\rangle$ in
		$\Gal(K'/F)$ has order $2$, so $\qf$ has ramification index $2$ in
		$K'/F$.
		
		Put $H=\Gal(M/K')$. As $[K':F]=m$ divides $\ell-1$, $H$ is the unique Sylow
		$\ell$-subgroup of $\Gal(M/F)$. Suppose that a divisor of $\qf$ in $K'$
		ramifies in $M$. Then $\#I=2\ell$, so $\#I_{1}=2$ and $I/I_{1}\cong\ZZ/\ell\ZZ$.
		Since $\operatorname{Aut}(I_{1})$ is trivial, $I_{1}$ is central in $I$; as
		$I/I_{1}$ is cyclic, $I$ is abelian, hence cyclic of order $2\ell$, and
		$H\subseteq I$. Let $t\in I$ be the element of order $2$. Then $t$ centralizes
		$H$, while $t\notin H$, so the image of $t$ in $\Gal(K'/F)$ is a nontrivial
		element acting trivially on $H$. This contradicts faithfulness.
	\end{proof}
	%checked!!
	
	\subsection{Work of Frey and Morrow}\label{ss:frey-morrow}
	In \cite{Frey}, Frey proved a double divisibility relation between $\ell$-torsion in class groups of imaginary quadratic fields and $\ell$-Selmer orders of twists of certain elliptic curves over $\mathbb{Q}$. This was extended to low degree number fields by Morrow \cite{Morrow}. The following version of \cite[Corollary~E]{Morrow} removes the restrictions on $[F:\QQ]$ and on $\Norm\qf$, and
	states the hypothesis on $\chi_H$ in the form of
	Proposition~\ref{prop:star}. We keep the notation of
		\cite{Morrow}; the remaining, minor differences from his statement are
		explained in the proof.
	
	\begin{theorem}\label{thm:morrow2}
		Let $F$ be a totally real Galois number field and let $\ell\geq5$ be a prime such that
		$\ell\nmid h(F)$, $\z_{\ell}\notin F$, and $m=[F(\z_{\ell}):F]$ is even.
		Let $E/F$ be an elliptic curve with an $F$-rational point $P$ of order
		$\ell$, and assume:
		\begin{enumerate}[label=\textup{(\roman*)},leftmargin=*]
			\item $P$ is not contained in the kernel of reduction modulo any prime
			$\lf\mid\ell$;
			\item every prime $\pf\mid N(E)$ with $\pf\nmid\ell$ and
			$\ell\nmid\ord_{\pf}(\Delta_{E})$ satisfies $\ell\mid\Norm\pf-1$.
		\end{enumerate}
		Let $d\in F^{\times}/(F^{\times})^{2}$ be totally negative and represented
		by an element of $\Ob$ prime to $\ell N(E)$, put $K=F(\sqrt d)$ and
		$\langle\d\rangle=\Gal(K/F)$, and suppose that $d$ satisfies the
		following divisibility and Artin symbol conditions:
		\begin{enumerate}[label=\textup{(\arabic*)},leftmargin=*]
			\item if $\qf\mid2$ and $\qf\mid N(E)$, then $\qf\mid D(K/F)$;
			\item if $\lf\mid\ell$ and $\ord_{\lf}(j_{E})<0$, then
			$\left(\dfrac{K/F}{\lf}\right)=\d$;
			\item for every prime $\pf\mid N(E)$ with $\pf\nmid2\ell$,
			\begin{equation*}
				\left(\dfrac{K/F}{\pf}\right)=
				\begin{cases}
					\delta, & \text{if } E/F_{\pf}\text{ is a Tate curve},\\[4pt]
					\operatorname{id}, & \text{otherwise}.
				\end{cases}
			\end{equation*}
		\end{enumerate}
		Then
		\begin{equation*}
			\#\Cl(K)[\ell]\ \Bigl|\ \#\Sel_{\ell}(E^{d},F)\ \Bigr|\
			\bigl(\#\Cl(K)[\ell]\bigr)^{2};
		\end{equation*}
		in particular $\Sel_{\ell}(E^{d},F)$ is nontrivial if and only if
		$\Cl(K)[\ell]$ is nontrivial.
	\end{theorem}

	\begin{proof}
		By Proposition~\ref{prop:star}, hypothesis \textup{(ii)} yields a
		primitive Hecke character $\chi_H$ of order $\ell$ with
		$\widetilde S_E(\chi_H)=\varnothing$. Hence $S_E(\chi_H)=\varnothing$,
		and conditions \textup{(1)}--\textup{(3)} are those of
		\cite[Theorem~B]{Morrow} with $S_E=\varnothing$, the branch
		$\ord_{\pf}(j_E)\geq0$ being vacuous by
		Corollary~\ref{cor:SE-equal}\textup{(a)}. Moreover $\ell\in S([F:\QQ])$,
		since $E(F)$ contains a point of order $\ell$.
		
		The hypotheses $\Norm\qf=2$ and $[F:\QQ]\leq5$ enter the proof of
		\cite[Theorem~B and Corollary~E]{Morrow} only through
		\cite[Lemma~5.5]{Morrow} (see \cite[Remark~5.1]{Morrow}), which
		Lemma~\ref{lem:dyadic-ram} replaces without either hypothesis. The
		proof of \cite{Morrow} then applies and gives the double
		divisibility.
	\end{proof}
	%checked!!
	\section{Main Theorem}\label{sec:main}
	Let $K$ be a CM field with maximal totally real subfield $K^{+}$, and let
	$\mu(K)$ denote its group of roots of unity. Put
	\begin{equation*}
		w_{K}=\#\mu(K),
		\qquad
		Q_{K}=\bigl[\,\Ob_{K}^{\times}:
		\mu(K)\,\Ob_{K^{+}}^{\times}\,\bigr].
	\end{equation*}
	The integer $Q_{K}\in\{1,2\}$ is the Hasse unit index of $K$.
	
	In our application $K=F(\sqrt{-2\x})$, so $K^{+}=F$. The quantities
	$Q_{F}(\sqrt{-2\x})$ and $w_{F}(\sqrt{-2\x})$ appearing in
	Theorem~\ref{thm:takai1} are $Q_{K}$ and $w_{K}$ in this notation.
	Since $\ell$ is odd and $[F(\z_{\ell}):F]\geq2g$, we have
	$\ell\nmid Q_{K}w_{K}$. Indeed, $Q_{K}\in\{1,2\}$, while
	$\ell\mid w_{K}$ would imply $\z_{\ell}\in K$, impossible since
	then $[F(\z_{\ell}):F]\leq[K:F]=2<2g$. Thus the denominator in
	\eqref{eq:takai-congruence} is invertible modulo $\ell$.
	
	Throughout this section, let $F/\mathbb{Q}$ be a totally real Galois number field of degree
	\begin{equation*}
		g=[F:\mathbb{Q}]\geq2.
	\end{equation*}

	%checked!!
	\subsection{Statement}
	
	To encode Morrow's local conditions in a form to which
Theorem~\ref{thm:takai2} applies, we choose quadratic characters
$\chi_i$ and signs $\varepsilon_i$ so that the required local behavior is
expressed by the conditions $\chi_i(2\x)=\varepsilon_i$. At primes above
$2$, the required characters are supplied by Lemma~\ref{lem:dyadic}.
The congruence \eqref{eq:takai-congruence} is verified for
Example~\ref{ex:intro} in \S\ref{sec:examples}.
	%checked!!
		\begin{theorem}\label{thm:lifting}
		Let $F$ be a totally real Galois number field of degree $g$, and let
		$\ell\geq5$ be a prime such that
		\begin{equation*}
			\ell\nmid h(F)
			\qquad\text{and}\qquad
			[F(\z_{\ell}):F]\geq2g .
		\end{equation*}
		Let $E/F$ be an elliptic curve with an $F$-rational point $P$ of order
		$\ell$, and assume:
		\begin{enumerate}[label=\textup{(\roman*)},leftmargin=*]
			\item $P$ is not contained in the kernel of reduction modulo any prime
			$\lf\mid\ell$, in the sense of \cite{Morrow};
			\item every prime $\pf\mid N(E)$ with $\pf\nmid\ell$ and
			$\ell\nmid\ord_{\pf}(\Delta_{E})$ satisfies $\ell\mid\Norm\pf-1$.
		\end{enumerate}
		Let
		\begin{equation*}
			\mathcal T=\{\pf:\ \pf\mid\ell N(E)\}=\{\pf_1,\dots,\pf_r\},
		\end{equation*}
		and for each $i$ define a primitive quadratic character $\chi_i$ and a
		sign $\varepsilon_i\in\{\pm1\}$ as follows.
		\begin{itemize}[leftmargin=*]
			\item If $\pf_i\mid2$, let $\chi_i$ and $\varepsilon_i$ be the character
			and the sign supplied by Lemma~\ref{lem:dyadic} for $\qf=\pf_i$.
			\item If $\pf_i\nmid2$, let $\chi_i=\left(\dfrac{\cdot}{\pf_i}\right)$
			and put
			\begin{equation*}
				\varepsilon_i=
				\begin{cases}
					-\left(\dfrac{-1}{\pf_i}\right),
					& \text{if } \pf_i\nmid\ell \text{ and } E/F_{\pf_i}
					\text{ is a Tate curve},\\[6pt]
					\phantom{-}\left(\dfrac{-1}{\pf_i}\right),
					& \text{if } \pf_i\nmid\ell \text{ and } E/F_{\pf_i}
					\text{ is not a Tate curve},\\[6pt]
					-\left(\dfrac{-1}{\pf_i}\right),
					& \text{if } \pf_i\mid\ell \text{ and } \ord_{\pf_i}(j_E)<0,\\[6pt]
					\text{either sign},
					& \text{if } \pf_i\mid\ell \text{ and } \ord_{\pf_i}(j_E)\geq0.
				\end{cases}
			\end{equation*}
		\end{itemize}
		Let $N_i$ be the conductor of $\chi_i$, set
		\begin{equation*}
			N\ZZ=N_1N_2\cdots N_r\cap\ZZ,
			\qquad
			A=\frac{gN^2D(F)}{8}
			\prod_{\substack{d\mid ND(F)\\ d\ \text{prime}}}
			\left(1+\frac1d\right),
		\end{equation*}
		and suppose there exists a rational prime $q>(A/g)^g$, unramified in
		$F/\QQ$, such that
		\begin{equation}\label{eq:takai-congruence}
			\sum_{\substack{
					\x\in2^{-1}\Ob_{+}\\
					\chi_i(2\x)=\varepsilon_i\ (i=1,\dots,r)\\
					\Tr(\x)=qg/2\\
					(q\Ob,\,2\x\Ob)\neq1}}
			\beta(\x)\,
			\frac{2^g\,h^{-}(F(\sqrt{-2\x})/F)}
			{Q_{F(\sqrt{-2\x})}\,w_{F(\sqrt{-2\x})}}
			\not\equiv0\pmod{\ell},
		\end{equation}
		with $\beta(\x)$ as in Theorem~\ref{thm:takai1}. Then
		\begin{equation*}
			\#\left\{
			d\in F^\times/(F^\times)^2:
			\begin{array}{l}
				d\ \text{totally negative},\\[3pt]
				\left|\Norm_{F/\QQ}\!\bigl(D(F(\sqrt d)/F)\bigr)\right|<X,\\[3pt]
				\Sel_\ell(E^d,F)=1
			\end{array}
			\right\}
			\gg_{F,E,\ell}
			\frac{X^{1/(2g)}}{\log X}.
		\end{equation*}
	\end{theorem}
	%checked!!
		
	\subsection{Proof of main result}\label{ss:proof}

		\begin{proof}
		Put $K_\x=F(\sqrt{-2\x})$ and $d_\x=-2\x$. By Theorem~\ref{thm:takai2},
		applicable since $[F(\z_{\ell}):F]\geq2g$ and the $\chi_i$ are primitive
		quadratic residue class characters, \eqref{eq:takai-congruence} implies
		that there are $\gg_{F, E, \ell}X^{1/(2g)}/\log X$ fields $K_\x$ with
		$|\Norm_{F/\QQ}(D(K_\x/F))|<X$, $\chi_i(2\x)=\varepsilon_i$ for all $i$,
		and $\ell\nmid h^{-}(K_\x/F)$. As $h(K_\x)=h(F)h^{-}(K_\x/F)$ and
		$\ell\nmid h(F)$, we get $\Cl(K_\x)[\ell]=1$. We verify the hypotheses of Theorem~\ref{thm:morrow2} for $d=d_\x$. The class $d_\x$ is totally negative, since $\x$ is totally positive. For every $\pf_i\in\mathcal T$ the condition $\chi_i(2\x)=\varepsilon_i\neq0$ forces $2\x$ to be a $\pf_i$-adic unit; hence $d_\x$ is represented by an element of $\Ob$ prime to $\ell N(E)$. For $\pf_i\nmid2$,
		\begin{equation*}
			\left(\dfrac{d_\x}{\pf_i}\right)
			=\left(\dfrac{-1}{\pf_i}\right)\chi_i(2\x)
			=\left(\dfrac{-1}{\pf_i}\right)\varepsilon_i,
		\end{equation*}
		so $\pf_i$ is inert in $K_\x/F$ when this is $-1$ and splits when it is $1$. By the choice of $\varepsilon_i$: if $\pf_i\nmid\ell$, then $\pf_i$ is inert when $E/F_{\pf_i}$ is a Tate curve and splits otherwise, which is condition \textup{(3)} of Theorem~\ref{thm:morrow2}; if $\pf_i\mid\ell$ and
		$\ord_{\pf_i}(j_E)<0$, then $\pf_i$ is inert, which is condition
		\textup{(2)}. For $\pf_i\mid2$, Lemma~\ref{lem:dyadic} shows that
		$\chi_i(2\x)=\varepsilon_i$ forces $\pf_i$ to ramify in $K_\x/F$, which is condition \textup{(1)}. Hypotheses \textup{(i)} and \textup{(ii)} of Theorem~\ref{thm:morrow2} are \textup{(i)} and \textup{(ii)} above, and $\z_{\ell}\notin F$ since $[F(\z_{\ell}):F]\geq2g\geq2$; this degree is even since $F$ is totally real. Theorem~\ref{thm:morrow2} therefore gives
		$\Sel_\ell(E^{d_\x},F)\neq1$ if and only if $\Cl(K_\x)[\ell]\neq1$, and
		the latter fails. Hence $\Sel_\ell(E^{d_\x},F)=1$ for every such $\x$,
		and the fields $K_\x$ correspond to distinct classes $d_\x$.
	\end{proof}
	%checked!!
	
	Theorem~\ref{thm:quadratic-version} follows from
	Theorem~\ref{thm:lifting} by taking $g=2$. Indeed,
	$S(2)\setminus\{2,3\}=\{5,7,11,13\}$, so these are the only primes
		$\ell\geq5$ that can occur. The degree $[F(\z_{\ell}):F]$ equals $\ell-1$, unless
		$F=\QQ(\sqrt{\ell})$ with $\ell\equiv1\pmod4$, in which case it equals
		$(\ell-1)/2$. Hence $[F(\z_{\ell}):F]\geq4=2g$ in all cases except $\ell=5$,
		$F=\QQ(\sqrt5)$, which is excluded.
	
	Let $q$ be as in Theorem~\ref{thm:quadratic-version}. The local
	conditions \textup{(1)}--\textup{(3)} give the character conditions
	$\chi_i(q)=\varepsilon_i$ occurring in
	Theorem~\ref{thm:lifting}, by the correspondence established in its
	proof; at the dyadic primes the required characters are supplied by
	the final assertion of Lemma~\ref{lem:dyadic} with $a_0=q$. For the signs that are unrestricted, take
	$\varepsilon_i=\chi_i(q)$.
	
	The dyadic characters may depend on $q$, but their conductors divide
	$\qf^{\,2e}$ by Lemma~\ref{lem:dyadic}. Thus the constant
	$C_{F,E,\ell}$ may be taken to be $(A/g)^g$, with
	$\qf^{\,2e}$ used in place of each dyadic conductor.
	
	Since $q$ is inert in $F$, Remark~\ref{rem:collapse} reduces
	\eqref{eq:takai-congruence} to
	\eqref{eq:takai-congruence-2}. As
	$\ell\nmid Q_{K}w_{K}$, this is equivalent to
	\begin{equation*}
		\ell\nmid h^{-}(K/F).
	\end{equation*}
	This proves Theorem~\ref{thm:quadratic-version}.
	\begin{remarknum}\label{rem:collapse}
		If $q$ is inert in $\Ob$ and $\chi_i(q)=\varepsilon_i$ for every $i$, then
		the sum in \eqref{eq:takai-congruence} consists of the single term
		$\x=q/2$. Hence \eqref{eq:takai-congruence} is equivalent to
		\begin{equation}
			\label{eq:takai-congruence-2}
			\frac{h^{-}(F(\sqrt{-q})/F)}
			{Q_{F(\sqrt{-q})}\,w_{F(\sqrt{-q})}}
			\not\equiv 0 \pmod{\ell}.
		\end{equation}
		Indeed, put $\nu=2\x\in\Ob_{+}$. Then $\Tr(\nu)=qg$, and since
		$q\Ob$ is prime, the condition $(q\Ob,\nu\Ob)\neq1$ gives
		$\nu=q\mu$ for some $\mu\in\Ob_{+}$. Thus $\Tr(\mu)=g$, and
		\begin{equation*}
			\Norm_{F/\QQ}(\mu)^{1/g}
			\leq
			\frac{\Tr(\mu)}{g}
			=
			1.
		\end{equation*}
		Since $\Norm_{F/\QQ}(\mu)$ is a positive integer, equality holds and
		$\mu=1$. Hence $\x=q/2$. Moreover, $q$ is unramified in $F$, so
		$\beta(q/2)=1$.
	\end{remarknum}
	%checked!!
	\section{Example}\label{sec:examples}
	Let $F=\QQ(\sqrt2)$, $a=\sqrt2$, and $\ell=7$. Then $F$ is totally real and
	Galois, $g=2$, $D(F)=8$, $h(F)=1$, and
	\begin{equation*}
		[F(\z_7):F]=6\geq4=2g,
	\end{equation*}
	since the quadratic subfield $\QQ(\sqrt{-7})$ of $\QQ(\z_7)$ is not $F$.
	Consider the global minimal model
	\begin{equation*}
		E:\ y^2+xy+ay=x^3+ax^2-(14+12a)x+24+18a,
	\end{equation*}
	with
	\begin{equation*}
		E(F)_{\mathrm{tors}}=\langle P\rangle\cong\ZZ/7\ZZ,
		\qquad
		P=(-a,-4-4a).
	\end{equation*}
	Put
	\begin{equation*}
		\qf=(a),\qquad
		\lf_1=(a-3),\qquad
		\lf_2=(a+3),\qquad
		\pf_{71}=(11-5a).
	\end{equation*}
	Then
	\begin{equation*}
		(2)=\qf^2,\qquad
		(7)=\lf_1\lf_2,\qquad
		N(E)=\qf\,\pf_{71},\qquad
		\Delta_{\min}(E)=\qf^{14}\pf_{71}.
	\end{equation*}
	The curve has split multiplicative reduction at $\qf$ and at $\pf_{71}$, and
	good reduction at $\lf_1$ and $\lf_2$.
	
	\emph{Hypothesis \textup{(i)}.} At $\lf_1$ and $\lf_2$ the reduction is good
	and $e(\lf_i/7)=1<6$, so $E_1(F_{\lf_i})$ has no $7$-torsion \cite[IV.6.1]{SilvermanAEC} and $P$ is not
	in the kernel of reduction modulo either prime. (A direct computation confirms that $P$ modulo
	$\lf_i$ has order $7$ for $i=1,2$.)
	
	\emph{Hypothesis \textup{(ii)}.} The primes $\pf\mid N(E)$ with
	$7\nmid\ord_{\pf}(\Delta_E)$ are those in
	$\{\pf:\ord_\pf(\Delta_{\min})\in\{1\}\}=\{\pf_{71}\}$, since
	$\ord_\qf(\Delta_{\min})=14$. As $\Norm\pf_{71}=71$ and $7\mid70$,
	hypothesis \textup{(ii)} holds. This agrees with Lemma~\ref{lem:redtype}: the
	reduction at $\pf_{71}$ is split, and at $\qf$, where $7\nmid\Norm\qf^2-1=3$,
	the exponent $14$ is divisible by $7$ as Corollary~\ref{cor:SE-equal}(b)
	predicts.
	
	\emph{Characters.} Here $\mathcal T=\{\qf,\pf_{71},\lf_1,\lf_2\}$. At $\qf$, a
	local computation gives the character of Lemma~\ref{lem:dyadic}, of exact
	conductor $\qf^{4}$, with sign $\varepsilon_\qf=1$; for rational odd $q$ the
	condition $\chi_\qf(q)=1$ reads $q\equiv1\pmod4$. At $\pf_{71}$,
	$\chi_{\pf_{71}}=\left(\frac{\cdot}{\pf_{71}}\right)$ and, $E/F_{\pf_{71}}$
	being a Tate curve,
	\begin{equation*}
		\varepsilon_{\pf_{71}}=-\left(\dfrac{-1}{\pf_{71}}\right)=1 .
	\end{equation*}
	At $\lf_1,\lf_2$, where $\ord_{\lf_i}(j_E)\geq0$, we take
	$\chi_{\lf_i}=\left(\frac{\cdot}{\lf_i}\right)$ with
	$\varepsilon_{\lf_i}=1$. Since both residue fields are $\mathbb F_7$, for
	rational $q$ the two conditions $\chi_{\lf_i}(q)=1$ coincide with
	$\left(\frac q7\right)=1$. Thus the character conditions on a rational odd
	prime $q$ are
	\begin{equation*}
		\left(\dfrac q7\right)=1,
		\qquad
		q\equiv1\pmod4,
		\qquad
		\left(\dfrac q{71}\right)=1 .
	\end{equation*}
	
	\emph{The constant.} The conductors are $\qf^{4}$, $\pf_{71}$, $\lf_1$, $\lf_2$,
	and
	\begin{equation*}
		\qf^{4}\lf_1\lf_2\pf_{71}\cap\ZZ=\qf^{4}\cdot(7)\cdot\pf_{71}\cap\ZZ=1988\,\ZZ,
	\end{equation*}
	so $N=1988$ and
	\begin{equation*}
		A=\frac{gN^2D(F)}{8}\prod_{d\mid ND(F)}\left(1+\frac1d\right)
		=13\,741\,056,
		\qquad
		(A/g)^g=6\,870\,528^2 .
	\end{equation*}
	
	\emph{The prime $q$.} Take $q=47\,204\,154\,999\,061$. Then $q>(A/g)^g$,
	\begin{equation*}
		q\equiv5\pmod8,\qquad
		\left(\dfrac q{71}\right)=1,\qquad
		\left(\dfrac q7\right)=1,
	\end{equation*}
	so $q$ is inert in $F$ and satisfies the three character conditions. By
	Remark~\ref{rem:collapse}, \eqref{eq:takai-congruence} reduces to
	$7\nmid h^{-}(K/F)$ for $K=F(\sqrt{-q})=\QQ(\sqrt2,\sqrt{-q})$. By
	\cite[Theorem~4.17]{Washington},
	\begin{equation*}
		h^{-}(K/F)=\frac{Q_K}{2}\,h(-4q)\,h(-8q),\qquad Q_K\in\{1,2\},
	\end{equation*}
	and \textsc{Pari}/GP, certified by \texttt{bnfcertify}, gives
	\begin{equation*}
		h(-4q)=2\,732\,182,\qquad h(-8q)=7\,439\,526,
	\end{equation*}
	neither divisible by $7$. Hence $7\nmid h^{-}(K/F)$, and
	Theorem~\ref{thm:lifting} gives
	\begin{equation*}
		\#\left\{
		d\in F^\times/(F^\times)^2:
		\begin{array}{l}
			d\text{ totally negative},\\[3pt]
			\left|\Norm_{F/\QQ}\!\bigl(D(F(\sqrt d)/F)\bigr)\right|<X,\\[3pt]
			\Sel_7(E^d,F)=1
		\end{array}
		\right\}
		\gg_{F,E,7}\frac{X^{1/4}}{\log X}. \ \\ \ 
	\end{equation*}
		%checked!!

	\bibliographystyle{alpha}
	\bibliography{references} 
\end{document}